\documentclass[11pt]{amsart}
\usepackage[margin=1in]{geometry}
\usepackage{bbm}
\usepackage{float, graphicx}
\usepackage[]{epsfig}
\usepackage{amsmath, amsthm, amssymb,}
\usepackage{epsfig}
\usepackage{verbatim}
\usepackage{multicol}
\usepackage{url}
\usepackage{latexsym}
\usepackage{mathrsfs}
\usepackage[colorlinks, bookmarks=true]{hyperref}
\usepackage{graphicx}
\usepackage{amsmath}
\usepackage{enumerate}
\usepackage[normalem]{ulem}
\usepackage{bm}
\usepackage{dsfont}
\usepackage{stmaryrd}
\usepackage{enumitem}
\usepackage{soul}
\usepackage[noadjust]{cite}
\usepackage{color}
\usepackage{xcolor}
\usepackage{hyperref}
\usepackage{indentfirst}
\usepackage{cleveref}
\usepackage{subcaption}
\usepackage{thmtools}

\numberwithin{equation}{section}
\usepackage{mathtools}

\newcommand{\bR}{{\mathbb R}}
\newcommand{\bS}{{\mathbb S}}

\newcommand{\bZ}{{\mathbb Z}}

\def\<{\langle}
\def\>{\rangle}

\def\Lamp[#1]{\boldsymbol{\Lambda}_{\mathrm{AMP}}^{(#1)}}
\def\lalg[#1]{\Lambda_{\mathrm{alg}, #1}}

\def\de{{\rm d}}

\def\calP{\mathcal{P}}

\renewcommand{\Im}{{\rm Im}}

\newcommand{\pa}{\partial}

\newcommand{\R}{\mathbb{R}}

\newcommand{\dist}{\operatorname{dist}}

\newcommand{\RN}[1]{%
  \textup{\uppercase\expandafter{\romannumeral#1}}%
}

\newcommand{\RNum}[1]{\uppercase\expandafter{\romannumeral #1\relax}}

\newcommand\du{\widetilde{u}}
\newcommand\dw{\widetilde{w}}
\newcommand\dphi{\widetilde{\phi}}
\theoremstyle{plain} %plain, definition, remark
\newtheorem{theorem}{Theorem}[section]
\newtheorem*{theorem*}{Theorem}
\newtheorem{lemma}[theorem]{Lemma}
\newtheorem*{lemma*}{Lemma}

\newtheorem*{corollary*}{Corollary}

\newtheorem*{proposition*}{Proposition}
\newtheorem{assumption}[theorem]{Assumption}
\newtheorem*{assumption*}{Assumption}

\newtheorem{definition}[theorem]{Definition}
\newtheorem*{definition*}{Definition}

\newtheorem*{example*}{Example}
\newtheorem{remark}[theorem]{Remark}

\newtheorem*{remark*}{Remark}
\newtheorem*{remarks*}{Remarks}

\title[Whispering Gallery Modes for Semilinear Dirichlet  Eigenvalue Problems]{Whispering Gallery Modes for Semilinear Dirichlet Eigenvalue Problems}

\date{\today}

\author[Z. Lin]{Zhengjiang Lin}
\address{(ZL) Department of Mathematics, Massachusetts Institute of Technology, 77 Massachusetts Ave, 02139 Cambridge MA, USA} %    Current address
\email{linzj@mit.edu}

\begin{document}

\maketitle

\begin{abstract}
We study the boundary localization phenomenon, known as whispering gallery modes, for weak solutions to semilinear Dirichlet eigenvalue problems in the unit ball $B_1 \subseteq \bR^d$ ($d \geq 2$) of the form
\[
\begin{cases}
-\Delta u + f(u) = \lambda u & \text{in } B_1,\\
u = 0 & \text{on } \partial B_1.
\end{cases}
\]
Here, $f = F'$ where $F$ is a nonnegative $C^2$-function with superquadratic polynomial growth. We prove the existence of a sequence of solutions $(u_n, \lambda_n)$ with $\lambda_n \to +\infty$ such that, for any $\tau \in (0,1)$,
\[
\lim_{n \to \infty} \frac{E_\tau(u_n)}{E_1(u_n)} = 0,
\]
where $E_\rho(u) = \int_{B_\rho} \bigl( \frac{1}{2} |\nabla u|^2 + F(u) \bigr) \, dx$ is the energy over the ball of radius $\rho$. This establishes that the energy of these high-eigenvalue solutions concentrates near the boundary, extending the classical whispering gallery mode phenomenon from linear Laplacian eigenfunctions to the semilinear setting. As a direct application, the case $F(u) = u^4 /4$ yields, after suitable scaling, a sequence of high-eigenvalue boundary-concentrating solutions, providing a nonlinear analogue of whispering gallery modes for the Allen-Cahn equations. The approach combines spectral properties of the linear Laplacian with local nonlinear bifurcation and is expected to adapt to related interior localization phenomena, such as bouncing ball modes, in other geometries where linear eigenfunctions exhibit analogous localization behaviors (e.g., filled ellipses or elliptical annuli).
\end{abstract}

\section{Introduction}

The whispering gallery phenomenon \cite{rayleigh1912problem, raman1922whispering, keller1960asymptotic,NG2013}, observed by Lord Rayleigh in the dome of St. Paul's Cathedral, captures the intuition that waves can remain confined near a curved boundary, propagating by shallow grazing reflections while decaying rapidly into the interior. In spectral geometry, this translates to high-frequency eigenfunctions of the Laplacian that concentrate arbitrarily close to the boundary of a smooth convex domain, a behavior known as whispering gallery modes.

Let $\Omega$ be a bounded domain in the Euclidean space $\bR^d$ with $d \geq 2$. For the linear Dirichlet eigenvalue problem $  -\Delta u = \lambda u  $ in $\Omega$ with $  u=0  $ on $  \partial \Omega  $, whispering gallery modes arise along a sequence of eigenvalues $  \lambda \to +\infty  $ for which the associated $  L^2  $-normalized eigenfunctions $  u_\lambda  $ satisfy, for any $  \tau > 0  $,
\begin{align*}
    \int_{\{ \dist(x,\partial\Omega) > \tau \}} |u_\lambda|^2 \, dx \to 0 \quad \text{as} \quad \lambda \to +\infty.
\end{align*}
This behavior was first rigorously established in general smooth strictly convex domains via semiclassical methods by Lazutkin and others \cite{babich1968eigenfunctions,lazutkin1968construction, lazutkin1973existence,lazutkin2012kam}. In the unit ball, as shown by \cite{NG2013}, separation of variables yields explicit radial profiles governed by spherical Bessel functions, whose uniform asymptotics confirm exponential decay away from the boundary for eigenfunctions corresponding to large eigenvalues, providing sharp quantitative estimates.

A natural and compelling question arises: does analogous boundary localization persist for \emph{nonlinear} eigenvalue problems of the form
\begin{align}\label{eq:eigen}
    -\Delta u + f(u) = \lambda u \quad \text{in } \Omega, \qquad u = 0 \quad \text{on } \partial \Omega,
\end{align}
where $  f = F'  $ and $  F \geq 0  $ is a $  C^2  $ potential with superlinear growth?  The present paper provides an affirmative answer by constructing, in the unit ball $  \Omega = B_1 \subseteq \mathbb{R}^d  $ ($  d \geq 2  $) with Dirichlet boundary conditions, a sequence of solutions to~\eqref{eq:eigen} whose energy localizes near the boundary as $  \lambda \to +\infty  $. This provides a genuine nonlinear analogue of whispering gallery modes and offers a mathematical explanation for the experimental observations reported in~\cite{harayama1999nonlinear}.

We pause to note that~\eqref{eq:eigen} is the Euler-Lagrange equation for critical points of the energy functional
$$E_\Omega(u) \coloneqq \int_\Omega \left( \frac{1}{2} |\nabla u(x)|^2 + F(u(x)) \right) \, dx$$
subject to $  L^2  $-mass constraints (with $  \lambda  $ serving as the Lagrange multiplier). Equations of this form arise in diverse contexts, including models of nonlinear scalar fields, reaction-diffusion systems, and phase transitions. For instance, when $  F(u) = u^4/4  $, a suitable rescaling of~\eqref{eq:eigen} recovers a form of the Allen--Cahn equation. The complex-valued counterpart of~\eqref{eq:eigen} is the Gross--Pitaevskii eigenvalue problem \cite{chen2025fully,chen2024convergence,atre2006class,rogel2013gross}, widely employed in quantum physics to describe Bose--Einstein condensation \cite{bose1924plancks,dalfovo1999theory,einstein2005quantentheorie}. Related boundary concentration phenomena in scaled versions of the Allen--Cahn equation (corresponding to $  \lambda \to +\infty  $) under Neumann boundary conditions were investigated in~\cite{malchiodi2007boundary,malchiodi2007boundary2,duan2021clustering,du2010interior}. In contrast, the regime $  \lambda \to -\infty  $ with Neumann conditions, which typically gives rise to interior or boundary spike solutions, has been far more extensively studied; see, e.g.,~\cite{malchiodi2005multiple, ni1995location, ni2006positive, lin2007number, ni1998location} and the references therein.
The present work complements these results by establishing continuous boundary concentration in the Dirichlet setting as $  \lambda \to +\infty  $, yielding solutions distinct from the discrete clustered-layer or spike morphologies prevalent in the Neumann literature.

Our main theorem is the following \Cref{thm:main whispering gallery}.

\begin{theorem}\label{thm:main whispering gallery}
    Let $\Omega = B_1$ be the unit ball in $\bR^d$ with $d \geq 2$. Assume that $F$ satisfies \Cref{a:superlinearity} and the Banach space $V$ is defined in \Cref{def:space V}. There is a family of solution pairs $\{ (u_{n},\lambda_n)\} \subseteq V \times \bR$ to \eqref{eq:eigen}, such that $
    \lim_{n \to \infty} \lambda_n =\infty$, and for any $\tau \in (0,1)$,
        \begin{align}
            \lim_{n \to \infty} \frac{E_\tau(u_n)}{E_1(u_n)} = 0,
        \end{align}
    where we let $B_{\tau}$ be the ball of radius $\tau$ and
        \begin{align}
            E_\tau(u) \coloneqq \int_{B_\tau} \frac{1}{2} |\nabla u(x)|^2 + F(u(x)) \ \de x .
        \end{align}
\end{theorem}

The proof of \Cref{thm:main whispering gallery} combines global spectral properties of the linear Laplacian with local nonlinear bifurcation, formalized in the following two auxiliary theorems. Although the present work focuses on real-valued functions, our methods extend directly to complex-valued settings, including the Gross--Pitaevskii eigenvalue problem.

\begin{theorem}\label{thm:main bifurcation}
    Adopt \Cref{a:superlinearity}, \Cref{def:space V}, and assume that the boundary of $\Omega$ is $C^1$-smooth. Let $\Lambda$ be a Dirichlet Laplacian eigenvalue of $\Omega$. There are positive constants $\delta_{\Lambda} ,\gamma_{\Lambda}$ depending on $\Lambda,\Omega,p_1,p_k,b,b_0$, such that for any $\lambda \in [\Lambda, \Lambda+\delta_{\Lambda}]$, there is a pair $(u_{\lambda}, \lambda) \in V \times \bR$ solving \eqref{eq:eigen}, satisfying
        \begin{enumerate}
            \item $u_{\lambda} = (\lambda - \Lambda)^{\frac{1}{p_1 - 2}} \left( w_{\lambda}+ \phi_{\lambda}\right)$ for two functions $w_{\lambda}, \phi_{\lambda}$ in $V$.
            \item $\Delta w_{\lambda} + \Lambda w_{\lambda} = 0$, $\|w_{\lambda}\|_2 \in [ \alpha \Lambda^{-\kappa},\beta] $, where $\alpha>0,\beta>0$ depends on $\Omega,p_1,p_k,b,b_0$ and $\kappa>0$ depends on $d,p_1,p_k$.
            \item $\phi_{\lambda}$ is orthogonal to $w_{\lambda}$ in $H_0 ^1 (\Omega)$,  and $\|\phi_{\lambda}\|_V \leq \gamma_{\Lambda} (\lambda - \Lambda)$.
        \end{enumerate}

\end{theorem}

The proofs for \Cref{thm:main bifurcation} adopt the technique of the bifurcation theory and Lyapunov-Schmidt reduction, widely used in consturction solutions to semilinear PDEs. See for example \cite{mugnai2005exact, malchiodi2005multiple, malchiodi2007boundary, crandall1973bifurcation,crandall1971bifurcation,correia2023note}.

The second key ingredient is a refined version of the classical whispering gallery phenomenon on unit balls for linear Laplacian eigenfunctions~\cite{rayleigh1912problem, raman1922whispering, keller1960asymptotic, NG2013, grebenkov2013geometrical,thanh2012localization}, adapted here to serve the nonlinear setting. In particular, the $L^p$ estimate in~\eqref{eq:main linear eigen 1} is contained in the proof of Theorem~2.1 in~\cite{NG2013}; our proof follows similar strategies, with ideas also drawn from~\cite{kuperman2026clamped}.

\begin{theorem}\label{thm:main whispering gallery linear eigen}
    Let $\Omega = B_1$ be the unit ball in $\bR^d$ with $d \geq 2$. There exists a sequence of Dirichlet Laplacian eigenvalues $\{\Lambda_n\}_{n=1} ^{\infty}$ of $\Omega$ such that the following holds: let $\{w_{\Lambda_n}\}_{n=1} ^{\infty}$ be any corresponding $L^2$-normalized eigenfunctions, that is, $w_{\Lambda_n} \in H_0 ^1 (\Omega)$, satisfying $\Delta w_{\Lambda_n} + \Lambda_n w_{\Lambda_n} = 0$ and $\|w_{\Lambda_n}\|_2 = 1$. Then, for any $p \geq 2$,
        \begin{align}\label{eq:main linear eigen 1}
            \int_{B_{\tau_{\Lambda_n}}} |\nabla w_{\Lambda_n}|^2 \de x \leq 2^{-\frac{1}{5} \Lambda_n ^{1/6}}, \quad  \int_{B_{\tau_{\Lambda_n}}} | w_{\Lambda_n}|^p \de x \leq 2^{-\frac{1}{10} \Lambda_n ^{1/6} \cdot p},
        \end{align}
    when $  \Lambda_n > C  $, where $  C = C(d) > 0  $ is a constant depending only on $  d  $. Here, $\tau_{\Lambda} \coloneqq 1- 2 \Lambda^{-1/6}$.
\end{theorem}

\Cref{thm:main whispering gallery linear eigen} mainly concerns boundary localization estimates for high-frequency Laplacian eigenfunctions on the ball. Analogous localization phenomena, either near the boundary or in the interior, and at low or high frequencies, have been extensively studied in more general domains \cite{babich1968eigenfunctions,lazutkin1968construction,lazutkin1973existence,lazutkin2012kam, jakobson2001geometric,sarnak2011recent,sapoval1991vibrations,even1999localizations,felix2007localization,heilman2010localized,delitsyn2012trapped,delitsyn2012exponential}. Beyond the linear Laplacian, similar behavior occurs for linear Schr\"{o}dinger operators of the form $\Delta w + V(x) w + \lambda w = 0$ with suitable potentials $V$ \cite{keller1960asymptotic}  and, more recently, for bi-Laplacian eigenfunctions of the form $\Delta^2 w - \lambda^2 w = 0$ under clamped boundary conditions at high frequencies \cite{kuperman2026clamped}.

We note that \Cref{thm:main whispering gallery} is established specifically for the unit ball $  \Omega = B_1  $, as it relies crucially on the quantitative localization estimates in \Cref{thm:main whispering gallery linear eigen}. In contrast, the local bifurcation result in \Cref{thm:main bifurcation} holds for arbitrary $  C^1  $-smooth domains. Accordingly, in more general geometries, including smooth strictly convex domains where high-frequency Laplacian eigenfunctions exhibit strong boundary concentration~\cite{keller1960asymptotic, babich1968eigenfunctions, lazutkin1968construction, lazutkin1973existence, lazutkin2012kam}, analogous nonlinear whispering gallery modes are expected to exist, contingent upon the availability of sufficiently refined quantitative estimates for the underlying linear Laplacian eigenfunctions.
Similar nonlinear localization phenomena are likewise anticipated in settings where linear eigenfunctions concentrate in the interior or display other forms of localization, such as bouncing ball modes in a filled ellipse or elliptical annulus~\cite{NG2013}(see also, e.g., \cite{jakobson2001geometric, sarnak2011recent, sapoval1991vibrations, even1999localizations, felix2007localization, heilman2010localized, delitsyn2012trapped, delitsyn2012exponential}). Therefore, our method provides a powerful tool for constructing and analyzing local concentration phenomena of semilinear equations on significantly more general domains.

To conclude the introduction, we prove \Cref{thm:main whispering gallery}, illustrating how it follows from \Cref{thm:main bifurcation} and \Cref{thm:main whispering gallery linear eigen}. The complete proofs of these two auxiliary theorems are given in \Cref{sec:proof of main bifurcation} and \Cref{sec:proof of main whispering gallery linear eigen}, respectively.

\begin{proof}[Proof of \Cref{thm:main whispering gallery}]
    Adopt the results of \Cref{thm:main whispering gallery linear eigen}, let $\Lambda$ be a Dirichlet Laplacian eigenvalue of $\Omega = B_1$ in the sequence $\{\Lambda_n\}_{n=1} ^{\infty}$ and $\Lambda > C$. Let $\gamma_{\Lambda}, \delta_{\Lambda}$ be the two constants in \Cref{thm:main bifurcation}. Set $\lambda = \Lambda + \delta  $ for some $\delta \in [0,\delta_{\Lambda}]$ to be determined later. Let $\sigma = \frac{1}{p_1 - 2}>0$. So the $u_{\lambda}$ in \Cref{thm:main bifurcation} is $  u_{\lambda} = \delta^{\sigma} (w_{\lambda} + \phi_{\lambda})  $ and $\|\phi_{\lambda}\|_V \leq \gamma_{\Lambda} \delta$. In the following, fix a $\tau \in (0,1)$, we estimate the ratio
        \begin{align}
            \frac{ E_\tau(u_{\lambda})}{\int_{B_1}  |\nabla u_{\lambda}(x)|^2  \ \de x }.
        \end{align}
    First, we estimate the denominator. We know that 
        \begin{align}
            \delta^{-2\sigma}|\nabla u_{\lambda}|^2 =  |\nabla w_{\lambda}|^2 + 2 \langle \nabla w_{\lambda},\nabla \phi_{\lambda} \rangle + |\nabla \phi_{\lambda}|^2  \geq \frac{1}{2} |\nabla w_{\lambda}|^2 - |\nabla \phi_{\lambda}|^2 ,
        \end{align}
    using the inequality $2 \langle \nabla w_{\lambda},\nabla \phi_{\lambda} \rangle \geq -\frac{1}{2} |\nabla w_{\lambda}|^2 -2|\nabla \phi_{\lambda}|^2$. Hence,
        \begin{align}
            \delta^{-2\sigma} \int_{B_1}  |\nabla u_{\lambda}(x)|^2  \ \de x \geq \frac{\Lambda}{2} \|w_{\lambda}\|_2 ^2 - (\gamma_{\Lambda} \delta)^2 \geq \frac{\alpha}{2}\Lambda^{1-\kappa} - (\gamma_{\Lambda} \delta)^2 \geq \frac{\alpha}{4}\Lambda^{1-\kappa},
        \end{align}
    where the second inequality follows from (2) of \Cref{thm:main bifurcation}, and the last inequality is true if we require that $\delta$ is small enough such that $(\gamma_{\Lambda} \delta)^2 < \frac{\alpha}{4}\Lambda^{1-\kappa}$. 

    Next, we estimate the Dirichlet energy term in the numerator. We again know that
        \begin{align}
            \delta^{-2\sigma}|\nabla u_{\lambda}|^2 \leq  2|\nabla w_{\lambda}|^2 + 2 |\nabla \phi_{\lambda}|^2.
        \end{align}
    Hence, by \Cref{thm:main whispering gallery linear eigen},
        \begin{align}
            \delta^{-2\sigma} \int_{B_{\tau}}  |\nabla u_{\lambda}(x)|^2  \ \de x \leq 2\int_{B_{\tau}}  |\nabla w_{\lambda}(x)|^2  \ \de x + 2 (\gamma_{\Lambda}\delta)^2 
                \leq 2  \cdot \|w_{\lambda}\|_2 ^2 \cdot 2^{-\frac{1}{5} \cdot \Lambda^{1/6}} + 2 (\gamma_{\Lambda}\delta)^2
                \leq 4  \beta ^2 \cdot 2^{-\frac{1}{5} \cdot \Lambda^{1/6}},
        \end{align}
    where in the last inequality, we used (2) in \Cref{thm:main bifurcation} that $\|w_{\lambda}\|_2 \leq \beta$ and require that $\delta$ is small enough such that $(\gamma_{\Lambda}\delta)^2 \leq \beta ^2 \cdot 2^{-\frac{1}{5} \cdot \Lambda^{1/6}}$.

    The remaining is the nonlinear term in the numerator. By \Cref{rem:f inequality} and the fact that $\delta^{-\sigma p} |u_{\lambda}|^p \leq 2^{p} |w_{\lambda}|^p + 2^{p} |\phi_{\lambda}|^p$ for $p \geq 2$, we have that
        \begin{align}
            \begin{split}
                \delta^{-\sigma p_1} \int_{B_{\tau}} F(u_{\lambda}) \de x &\leq C \int_{B_{\tau}}\left( |w_{\lambda}|^{p_1} + |\phi_{\lambda}|^{p_1} + \delta^{\sigma (p_k - p_1)}\left(|w_{\lambda}|^{p_k} + |\phi_{\lambda}|^{p_k}\right)\right) \de x
            \\  &\leq C \int_{B_{\tau}}\left( |w_{\lambda}|^{p_1} + |\phi_{\lambda}|^{p_1} + \left(|w_{\lambda}|^{p_k} + |\phi_{\lambda}|^{p_k}\right)\right) \de x
            \end{split}
        \end{align}
    for a constant $C$ depending on $b,p_1,p_k$, where we used the fact that $\delta^{\sigma (p_k - p_1)} \leq 1$. For the term $\int_{B_{\tau}} |\phi_{\lambda}|^{p_k} \de x$, we have that
        \begin{align}
            \int_{B_{\tau}} |\phi_{\lambda}|^{p_k} \de x \leq \|\phi_{\lambda}\|_{p_k} ^{p_k} \leq C \|\phi_{\lambda}\|_q ^{p_k} \leq C \|\phi_{\lambda}\|_V ^{p_k} \leq C (\gamma_{\Lambda} \delta)^{p_k},
        \end{align}
    where we used H\"older inequality in the second inequality and $C$ is a constant depending on $p_k$ and $d$. We can similarly get that $\int_{B_{\tau}} |\phi_{\lambda}|^{p_1} \de x \leq C (\gamma_{\Lambda} \delta)^{p_1}$. For the term $\int_{B_{\tau}} |w_{\lambda}|^{p_k} \de x$, by \Cref{thm:main whispering gallery linear eigen} and (2) in \Cref{thm:main bifurcation}, we have that
        \begin{align}
            \int_{B_{\tau}} |w_{\lambda}|^{p_k} \de x \leq 2^{-\frac{1}{10} \Lambda^{1/6} \cdot p_k} \|w_{\lambda}\|_2 ^{p_k} \leq \beta^{p_k} 2^{-\frac{1}{10} \Lambda^{1/6} \cdot p_k} \leq \beta^{p_k} 2^{-\frac{1}{5} \Lambda^{1/6} },
        \end{align}
    where the last inequality is because $p_k > 2$. Similarly, we have that $\int_{B_{\tau}} |w_{\lambda}|^{p_k} \de x \leq \beta^{p_1} 2^{-\frac{1}{5} \Lambda^{1/6} }$. Combine the above estimates, after choosing $\delta$ small enough such that $C (\gamma_{\Lambda} \delta)^{p_1} \leq \beta^{p_k} 2^{-\frac{1}{5} \Lambda^{1/6} }$, we have that
        \begin{align}
            \delta^{-\sigma p_1} \int_{B_{\tau}} F(u_{\lambda}) \de x \leq C 2^{-\frac{1}{5} \Lambda^{1/6} },
        \end{align}
    where $C$ is a constant depending on $b,b_0,p_1,p_k,d$ and is independent of $\Lambda$.

    Combine the estimates for the numerator and the denominator, we have that
        \begin{align}
            \frac{ E_\tau(u_{\lambda})}{\int_{B_1}  |\nabla u_{\lambda}(x)|^2  \ \de x } \leq C \Lambda^{\kappa - 1}(1+\delta^{\sigma(p_1 - 2)})2^{-\frac{1}{5} \Lambda^{1/6} } \leq 2C\Lambda^{\kappa - 1} 2^{-\frac{1}{5} \Lambda^{1/6} },
        \end{align}
    where $C$ is a constant depending on $b,b_0,p_1,p_k,d$ and is independent of $\Lambda$, and we used the fact that $\delta^{\sigma(p_1 - 2)} <1$ because $\delta <1$ and $p_1 >2$. Hence, after requiring that $\delta < \delta_{\Lambda} '$ for some $\delta_{\Lambda} '$ depending on the above arguments (depending on $\Lambda,d,p_1,p_k,b,b_0$), we have that 
        \begin{align}
            \frac{ E_\tau(u_{\lambda})}{\int_{B_1}  |\nabla u_{\lambda}(x)|^2  \ \de x } \leq 2C\Lambda^{\kappa - 1} 2^{-\frac{1}{5} \Lambda^{1/6} }.
        \end{align}
    Hence, we pick $\{\Lambda_n\}_{n=1} ^{\infty}$ as the sequence in \Cref{thm:main whispering gallery linear eigen}, and pick a corresponding $\delta_n < \delta_{\Lambda_n} ' $ to get a $\lambda_n =\Lambda_n + \delta_n$. The corresponding $\{u_{\lambda_n}\}_{n=1} ^{\infty}$ then satisfies that
        \begin{align}
            \frac{ E_\tau(u_{\lambda_n})}{E_1(u_{\lambda_n})} \leq \frac{ E_\tau(u_{\lambda_n})}{\frac{1}{2}\int_{B_1}  |\nabla u_{\lambda_n}(x)|^2  \ \de x } \leq 4C\Lambda_n ^{\kappa - 1} 2^{-\frac{1}{5} \Lambda_n ^{1/6} },
        \end{align}
    which goes to $0$ exponentially fast as $n \to +\infty$.
\end{proof}

\section{Preliminary}\label{sec:preliminary}
In this \Cref{sec:preliminary}, we record the assumptions on the potential $  F  $ and introduce the functional setting in which our solutions will be constructed. We impose a specific superquadratic polynomial growth condition on $  F  $, which encodes the superlinear character of the nonlinearity $  f=F'  $. We then define the natural Banach space $  V  $ adapted to this growth and dimension $  d  $, together with the associated Sobolev embeddings and interpolation inequalities that will be invoked repeatedly in the bifurcation analysis in \Cref{sec:proof of main bifurcation}.

\begin{assumption}\label{a:superlinearity}
    $F(s)$ is nonnegative and has the following form: for any $s \in \bR$,
        \begin{align}\label{eq:superlinearity}
            F(s) = b_1 |s|^{p_1} + b_2 |s|^{p_2} + \dots + b_k |s|^{p_k},
        \end{align}
    where $k \in \bZ_+$, $2<  p_1<p_2, \dots < p_k$, and $\{b_i\}_{i=1} ^{k}$ are constants in $\bR$. Also, we assume that $F(s) \geq b_0 |s|^{p_1}$ for some $b_0 > 0$.
\end{assumption}

\begin{remark}\label{rem:positivity F}
    In \Cref{a:superlinearity}, because $F(s)$ has the polynomial form \eqref{eq:superlinearity}, the extra assumption that $F(s) \geq b_0 |s|^{p_1}$ for some $b_0 >0$ is equivalent to that $F(s) > 0$ unless $s=0$. \Cref{thm:main bifurcation} can also be generalized to the case when assuming $F(s) < 0$ unless $s = 0$, following the same proof as in \Cref{sec:proof of main bifurcation}. The only necessary difference is now the choice of $\lambda$ is within $[\Lambda-\varepsilon_{\Lambda} , \Lambda]$, because of the proof of \Cref{lem:critical point for tilde G}.
\end{remark}

\begin{remark}
    One can replace \Cref{a:superlinearity} with suitable polynomial growth inequalities assumptions on $f',f,F$, that is, $|f'| \asymp s^{p}$, $|f| \asymp s^{p+1}$, $|F| \asymp s^{p+2}$, and obtain results similar to \Cref{thm:main whispering gallery} and \Cref{thm:main bifurcation}. Also, one can replace those constants $b_i$'s with bounded functions on $\Omega$, and \Cref{thm:main whispering gallery} and \Cref{thm:main bifurcation} still persist. To simplify the arguments and elucidate the main ideas clearly, we keep this current \Cref{a:superlinearity}.
\end{remark}
\begin{remark}\label{rem:f inequality}
     In the proofs, when a constant depends on the upper bounds of $|b_1|,\dots,|b_k|$, we say this constant depends on $b$. A inequality we will use frequently is the following: there are constants $C,C'$ depending on $b,p_1,p_k$, such that $|F(s)| \leq C |s|^{p_1} + C' |s|^{p_k}$, $|f(s)| = |F'(s)| \leq C |s|^{p_1 - 1} + C' |s|^{p_k-1}$, and $|f'(s)| = |F''(s)| \leq C |s|^{p_1 - 2} + C' |s|^{p_k-2}$ by interpolation.
\end{remark}

\begin{comment}
When $F(0) = 0$ but $F$ change sign, if $F$ still has the polynomial form \eqref{a:superlinearity}, then $F$ must reach a local extremum at some $s > 0$, which is also a critical point of $F$. $f = F'$ also changes sign at $s$. Hence, we make the following assumption on $F$ which is different from \Cref{a:superlinearity}.

\begin{assumption}\label{a:superlinearity critical}
    $F(s)$ has the form \eqref{a:superlinearity}. Also, we assume that $F$ has a local extremum at an $s_0 >0$ and $f$ changes sign at $s_0$. Without loss of generality, we assume that $s_0 \in (0,\frac{1}{2})$.
\end{assumption}
\end{comment}

Next, we discuss the solutions space for \eqref{eq:eigen} we work with.

\begin{definition}\label{def:space V}
    Adopt \Cref{a:superlinearity}. We define the Banach space 
    \begin{align}\label{eq:space V}
        V \coloneqq
        \begin{cases}
             H_0 ^1 (\Omega), &\text{ if } d = 2 \text{ or } d \geq 3 \text{ and } p_k \leq \frac{2d}{d-2}, 
            \\ H_0 ^1 (\Omega) \cap L^q (\Omega) , \text{ where } q = \frac{d(p_k -2)}{2}, &\text{ if } d \geq 3 \text{ and } p_k > \frac{2d}{d-2} . 
        \end{cases}
    \end{align}
The choice of $q$ is to make sure $\frac{(p_k - 1) d q}{d + 2q} = q$. Also, this $q$ satisfies that $q > \frac{2d}{d-2}$ and $q > p_k$.
     
     The corresponding norms $\| u\|_V$ for $u \in V$ are defined as 
     \begin{align}\label{eq:norm V}
        \| u \|_V \coloneqq
        \begin{cases}
             \|u\| , &\text{ if } d = 2 \text{ or } d \geq 3 \text{ and } p_k \leq \frac{2d}{d-2}, 
            \\ \|u\| + \|u\|_q, &\text{ if } d \geq 3 \text{ and } p_k > \frac{2d}{d-2} . 
        \end{cases}
    \end{align}
    Throughout this paper, we use the notations 
        \begin{align}
            \|u\| \coloneqq \left( \int_{\Omega} |\nabla u |^2 \de x\right) ^{\frac{1}{2}}, \quad \|u\|_r \coloneqq \left( \int_{\Omega} | u |^r \de x\right) ^{\frac{1}{r}}, \text{ for any }r \geq 1.
        \end{align}
\end{definition}

\begin{definition}\label{def:sobolev embedding}
    Let $\|\cdot\|$ be the norm on $H_0 ^1 (\Omega)$. When $d \geq 3$, we let $i$ be the Sobolev embedding map $i: H_0 ^1 (\Omega) \hookrightarrow L^{\frac{2d}{d-2}}(\Omega) $, and the dual map is $i^{\ast} : L^{\frac{2d}{d+2}} \to H_0 ^1 (\Omega)$; when $d = 2$, $i: H_0 ^1 (\Omega) \hookrightarrow \cap_{r > 1} L^{r}(\Omega) $, and the dual map is $i^{\ast} : \cup_{r > 1} L^{r}(\Omega) \to H_0 ^1 (\Omega)$.
\end{definition}

When $d \geq 3$, for any $v \in H_0 ^1 (\Omega)$ and $u \in L^{\frac{2d}{d+2}}$, we see that
    \begin{align}\label{eq:i ast sobolev norm}
        \int_{\Omega} \langle \nabla i^{\ast}(u) , \nabla v \rangle \de x = \int_{\Omega} u \cdot i(v) \de x \leq C \|u \|_{\frac{2d}{d+2}} \|i(v) \|_{\frac{2d}{d-2}} \leq C \|u \|_{\frac{2d}{d+2}} \|v\|,
    \end{align}
where we used the Sobolev inequality.
Thus, $\| i^{\ast}(u) \| \leq C \|u \|_{\frac{2d}{d+2}}$, and we see that $i^{\ast} = (-\Delta)^{-1}$, the inverse of the negative Laplacian operator. Here, $C$ is a constant only depending on $\Omega$. Similarly, when $d = 2$, we have that for any $r > 1$, $\| i^{\ast}(u) \| \leq C(r) \|u \|_{r}$, and $C(r)$ is a constant only depending on $\Omega,r$. Also, we have the following lemma for the case when $d \geq 3$ and $p_k > \frac{2d}{d-2}$. See for example \cite{agmon1959estimates}, Lemma 9.17 and Theorem 7.26 of \cite{GT77}, or Lemma 2.2 of \cite{mugnai2005exact}.

\begin{lemma}\label{lem:s norm inequality 1}
    Let $d \geq 3$ and $s > \frac{2d}{d-2}$. If $u \in L^{\frac{ds}{d+2s}}(\Omega) \subseteq L^{\frac{2d}{d+2}}(\Omega)$, then $i^{\ast}(u) \in L^s(\Omega)$, and there is a positive constant $C$ depending on $d,s,\Omega$, such that $\|i^{\ast}(u)\|_s \leq C \|u\|_{\frac{ds}{d+2s}}$.
\end{lemma}
By \Cref{lem:s norm inequality 1} and the choice of $q$ in \Cref{def:space V}, we see that for $u \in L^{q} (\Omega)$, $\|i^{\ast}(u^{p_k})\|_q \leq C \|u^{p_k}\|_{\frac{dq}{d+2q}} = C \|u\|_{q} ^{p_k}$. We will need this fact in \Cref{sec:proof of main bifurcation}.

%-------------------------------
\section{Proof of \Cref{thm:main bifurcation}}\label{sec:proof of main bifurcation}

We are going to prove \Cref{thm:main bifurcation} by perturbing a solution $w_{\lambda}$ solving $\Delta w_{\lambda} + \Lambda w_{\lambda} = 0$ using the bifurcation theory and Lyapunov-Schmidt reduction. See for example \cite{mugnai2005exact, malchiodi2005multiple, malchiodi2007boundary, crandall1973bifurcation,crandall1971bifurcation,correia2023note} for the application of bifurcation theory and Lyapunov-Schmidt reduction in constructing solutions to semilinear PDEs. The main idea is the following: if $u \in V$ is a solution to  \eqref{eq:eigen}, then $\du \coloneqq Mu \in V$ with some positive constant $M$ becomes a solution to
\begin{align}\label{eq:eigen rescaling}
        \begin{cases}
            \Delta \du + \lambda \du - M f\left(M^{-1}\du \right) = 0, &\text{ in } \Omega, 
            \\ \du=0 , &\text{ on } \pa \Omega .
        \end{cases}
\end{align}
Because $f$ is superlinear, we can increase the scale $M$ so that the term $M f\left(M^{-1}\du \right)$ in \eqref{eq:eigen rescaling} becomes negligible. Hence, we are able to obtain a solution to \eqref{eq:eigen rescaling} by perturbing a Dirichlet Laplacian eigenfunction using the bifurcation theory and Lyapunov-Schmidt reduction.

To elucidate the above ideas rigorously, we denote the Laplacian eigenspace of the eigenvalue $\Lambda$ as $K_{\Lambda}$, which has a finite dimension $N_{\Lambda} \in \bZ_+$. That is, $K_{\Lambda}$ consists of $f \in H_0 ^1 (\Omega)$ satisfying $\Delta f + \Lambda f = 0$. Because $H_0 ^1 (\Omega)$ is a Hilbert space with respect to the inner product $\langle u, v \rangle \coloneqq  \int_{\Omega} \langle \nabla u , \nabla v \rangle \de x$, we denote the orthogonal complement of $K_{\Lambda}$ in $H_0 ^1 (\Omega)$ as $K_{\Lambda} ^{\perp}$. Define $\calP: H_0 ^1 (\Omega) \to K_{\Lambda}$ as the orthogonal projection onto $K_{\Lambda}$ and $\calP^{\perp}: H_0 ^1 (\Omega) \to K_{\Lambda} ^{\perp}$ as the orthogonal projection onto $K_{\Lambda} ^{\perp}$. Notice that if we let $\{e_i\}_{i=1} ^{N_{\Lambda}}$ be an orthonormal basis of $K_{\Lambda}$, we can explicitly write $P(u) = \sum_{i=1} ^{N_\Lambda} \langle u , e_i \rangle e_i$, and $P^{\perp} (u) = u - P(u)$ for any $u \in H_0 ^1 (\Omega)$. By the definition of $\|\cdot \|_V$ in any cases in \eqref{eq:norm V}, we can use this explicit formula of $P(u)$ to obtain constants $C(\Lambda)$ depending on $\Lambda$, such that 
    \begin{align}\label{eq:P,P perp norm}
        \|P(u)\|_V \leq C(\Lambda) \|u\|_V, \quad \|P^{\perp} (u)\|_V \leq C(\Lambda) \|u\|_V .
    \end{align}
So, we can decompose $V$ in \Cref{def:space V} as $V = K_{\Lambda} \oplus (K_{\Lambda} ^{\perp} \cap V)$ and we would like to search solutions to \eqref{eq:eigen rescaling} of the form $\du = \dw_\lambda + \dphi_\lambda$ for $\dw_\lambda \in K_{\Lambda}$ and $\dphi_\lambda \in (K_{\Lambda} ^{\perp} \cap V)$. We first see that \eqref{eq:eigen rescaling} is equivalent to 
    \begin{align}\label{eq:eigen rescaling 0}
        \du - i^{\ast} \left(  \lambda \du - M f\left(M^{-1}\du \right) \right) = 0 ,
    \end{align}
which can be decomposed into
    \begin{align}\label{eq:eigen rescaling 1}
        P^{\perp} \left( \du - i^{\ast} \left(  \lambda \du - M f\left(M^{-1}\du \right) \right) \right) = 0 ,
    \end{align}
and 
    \begin{align}\label{eq:eigen rescaling 2}
        P \left( \du - i^{\ast} \left(  \lambda \du - M f\left(M^{-1}\du \right) \right) \right) = 0 .
    \end{align}
Using the decomposition $\du = \dw_\lambda + \dphi_\lambda$ for $\dw_\lambda \in K_{\Lambda}$ and $\dphi_\lambda \in (K_{\Lambda} ^{\perp} \cap V)$, and the fact that $i^{\ast} \dw_{\lambda} = (-\Delta)^{-1} \dw_{\lambda} = \Lambda^{-1} \dw_{\lambda}$ and $i^{\ast}(K_{\Lambda} ^{\perp}) \subseteq K_{\Lambda} ^{\perp}$, \eqref{eq:eigen rescaling 1} and \eqref{eq:eigen rescaling 2} are reduced to
    \begin{align}\label{eq:eigen rescaling 1 simplified}
         \dphi_{\lambda} -  \lambda i^{\ast} \dphi_{\lambda}   + M P^{\perp}  i^{\ast} \left[f\left(M^{-1}\du \right) \right] = 0 ,
    \end{align}
and 
    \begin{align}\label{eq:eigen rescaling 2 simplified}
         \left(1-\frac{\lambda}{\Lambda} \right)\dw_{\lambda}+ M P  i^{\ast} \left[f\left(M^{-1}\du \right) \right]  = 0 .
    \end{align}
In the following, we solve \eqref{eq:eigen rescaling 1 simplified} and \eqref{eq:eigen rescaling 2 simplified} respectively.

\begin{lemma}\label{lem:solve simplified 1}
    There exists a positive constant $\delta_{\Lambda}$ depending on $\Lambda,\Omega$ and a positive constant $C_{\Lambda}$ depending on $\Lambda,p_1,p_k,\Omega$, such that for any $ W > 1 $ and for any $w$ with $\|w\|_V \leq W$, if $|\lambda - \Lambda| < \delta_{\Lambda}$ and $M > C_{\Lambda} W^{\frac{p_1 - 1}{p_1-2}}$, then there exists a unique $\dphi_\lambda \in (K_{\Lambda} ^{\perp} \cap V)$ solving \eqref{eq:eigen rescaling 1 simplified} with $\du = w + \dphi_\lambda$, and satisfying $\|\dphi_{\lambda}\|_V \leq M^{2-p_1} C_{\Lambda} ^{p_1 - 2} W^{p_1 -1} $. This $\dphi_{\lambda}$ is $C^1$-differentiable with respect to $w \in K_{\Lambda}$ and we denote it as $\dphi_{\lambda}(w)$.
\end{lemma}

\begin{lemma}\label{lem:solve simplified 2}
    There is a positive constant $C_{\Lambda} '$ depending on $\Lambda,\Omega,p_1,p_k,b,b_0$, a positive constant $\eta<1$ depending on $b_0,\Omega,p_1$, a constant $\varepsilon_{\Lambda}<\frac{1}{2}$ of the form $\varepsilon_{\Lambda} = C\Lambda^{-\kappa} \eta^{\frac{1}{p_1 - 2}}$ for $C>0$ depending on $\Omega,p_1,p_k,b,b_0$ and $\kappa>0$ depending on $d,p_1,p_k$, such that if $M > C_{\Lambda} '$ and if we let $\lambda = \Lambda + 2 \eta M^{2-p_1}$, there exists a $\dw_{\lambda} \in K_{\Lambda}$ satisfying $\|\dw_{\lambda}\|_2 \in (\frac{\varepsilon_{\Lambda}}{2} , \frac{1}{2})$, such that the function $\du = \dw_\lambda + \dphi_\lambda (\dw_{\lambda})$ is a solution to \eqref{eq:eigen rescaling}, where $\dphi_\lambda (\dw_{\lambda})$ is obtained in \Cref{lem:solve simplified 1} and satisfies $\|\dphi_{\lambda}\|_V \leq M^{2-p_1} (C_{\Lambda} ')  ^{p_1 - 2}$.
\end{lemma}

\subsection{Proof of \Cref{lem:solve simplified 1}}
Define a bounded linear operator 
    \begin{align}\label{eq:def Alambda}
        A_{\lambda}: K_{\Lambda} ^{\perp} \to K_{\Lambda} ^{\perp}, \ \phi \mapsto \phi - \lambda i^{\ast} \phi.
    \end{align}
By \Cref{lem:s norm inequality 1}, we see that there is a positive constant $C$ depending on $d,p_k,\Lambda,\Omega$, such that $\|A_{\lambda}(\phi)\|_V \leq C \|\phi\|_V$ if $\phi \in V $. Hence, $\Im \left(A_{\lambda} |_{V \cap K_{\Lambda} ^{\perp}}  \right) \subseteq V \cap K_{\Lambda} ^{\perp}$. 
The following \Cref{lem:invertibility eigengap} shows that $A_{\lambda}|_{V \cap K_{\Lambda} ^{\perp}}$ is a invertible map on $V \cap K_{\Lambda} ^{\perp}$.

\begin{lemma}\label{lem:invertibility eigengap}
    There exist a positive constant $\delta_{\Lambda}$ depending on $\Lambda, \Omega$ and a positive constant $c_{\Lambda}$ depending on $\Lambda, p_k,  \Omega$, such that for any $|\lambda - \Lambda| < \delta_{\Lambda}$ and any $\phi \in V \cap K_{\Lambda} ^{\perp}$,
        \begin{align}\label{eq:invertibility eigengap V}
            \| A_{\lambda}(\phi) \|_V \geq c_{\Lambda} \| \phi \|_V.
        \end{align}
\end{lemma}
\begin{proof}
    We only prove the lemma when $d \geq 3 $ and $p_k > \frac{2d}{d-2}$, because the proof for the other case in \Cref{def:space V} is the same. We first show that there are $\delta_{1},c_{1}$ depending on $\Lambda, \Omega$, such that for any $|\lambda - \Lambda| < \delta_{1}$ and any $\phi \in K_{\Lambda} ^{\perp}$,
        \begin{align}\label{eq:invertibility eigengap H1}
            \| A_{\lambda}(\phi) \| \geq c_{1} \| \phi \|.
        \end{align}
    Assume that $\{\varphi_j\}$ is an orthonormal basis of $K_{\Lambda} ^\perp$ in $H_0 ^1 (\Omega)$ consisting of Dirichlet Laplacian eigenfunctions of eigenvalues $\Lambda_j \neq \Lambda$. Hence, because $i^{\ast} = (-\Delta)^{-1}$, we have that
        \begin{align}
            \phi = \sum_{j} \langle \varphi_j , \phi \rangle \varphi_j, \text{ and } A_{\lambda} = \sum_{j} \langle \varphi_j , \phi \rangle  \left(1-\frac{\lambda}{\Lambda_j} \right)\varphi_j.
        \end{align}
    Hence,
        \begin{align}
            \|A_{\lambda} \phi\| ^2 = \sum_{j} |\langle \varphi_j , \phi \rangle|^2  \left(1-\frac{\lambda}{\Lambda_j} \right)^2 \geq \min_{j} \left(1-\frac{\lambda}{\Lambda_j} \right)^2 \cdot \|\phi\|^2.
        \end{align}
    So, if we list all Dirichlet Laplacian eigenvalues in the order $0<\Lambda_1<\dots < \Lambda_- < \Lambda < \Lambda_+ < \dots$, and if we let $\delta_1 = \frac{1}{2} \min\{\Lambda_+ - \Lambda, \Lambda - \Lambda_-\}$, we see that when $|\lambda -\Lambda| < \delta_1$, we have that $\min_{j} \left(1-\frac{\lambda}{\Lambda_j} \right)^2 \geq c_1 ^2 > 0$ for a positive constant $c_1$ depending on $\Lambda_-, \Lambda , \Lambda_+$.

    Next, by the definition of $A_{\lambda}$ \eqref{eq:def Alambda},
        \begin{align}
            \|\phi\|_q = \|A_{\lambda}(\phi) + \lambda i^{\ast} (\phi)\|_q \leq  \|A_{\lambda}(\phi)\|_q + \lambda \| i^{\ast} (\phi)\|_q. 
        \end{align}
    Using \Cref{lem:s norm inequality 1} for $s = q$, we see that $\| i^{\ast} (\phi)\|_q \leq C \|\phi\|_{\frac{q}{p_k}} \leq C \|\phi\|_{q} ^{1-\alpha} \|\phi\|_1 ^{\alpha}$, where the last inequality follows from the H{\"o}lder inequality and $\alpha \in (0,1)$ satisfies that $\frac{p_k}{q} = \frac{1-\alpha}{q} + \frac{\alpha}{1}$. Hence, apply the Sobolev inequality for $\|\phi\|_1$, we obtain that there is a constant $C$ depending on $d,q,\Omega$, such that $\| i^{\ast} (\phi)\|_q \leq C \|\phi\|_{q} ^{1-\alpha} \|\phi\| ^{\alpha}$, and thus 
        \begin{align}
            \|\phi\|_q \leq \|A_{\lambda}(\phi) \|_q +  C \|\phi\|_{q} ^{1-\alpha} \|\phi\| ^{\alpha} ,
        \end{align}
    for a constant $C$ depending on $\Lambda, d,q,\Omega$.
    
    When $ C\|\phi\|^{\alpha} \leq \frac{1}{2} \|\phi\|_q ^{\alpha}$, we see that $ \|\phi\|_q \leq 2\|A_{\lambda}(\phi) \|_q$, and hence $\| A_{\lambda}(\phi) \|_V \geq \min\{c_1 , \frac{1}{2} \} \| \phi \|_V$ after combining \eqref{eq:invertibility eigengap H1}. When $ C\|\phi\|^{\alpha} > \frac{1}{2} \|\phi\|_q ^{\alpha}$, i.e., $ (2C)^{\frac{1}{\alpha}} \|\phi\| >  \|\phi\|_q $, we see that 
        \begin{align}
            \|\phi\|_V \leq (1+(2C)^{\frac{1}{\alpha}})\|\phi\| \leq c_1 ^{-1} (1+(2C)^{\frac{1}{\alpha}}) \| A_{\lambda}(\phi) \| \leq c_1 ^{-1} (1+(2C)^{\frac{1}{\alpha}}) \| A_{\lambda}(\phi) \|_{V}
        \end{align}
    by \eqref{eq:invertibility eigengap H1}. In both cases, we obtain a $c_{\Lambda}$ only depending on $\Lambda, p_k,  \Omega$ such that $\| A_{\lambda}(\phi) \|_V \geq c_{\Lambda} \| \phi \|_V$ when we pick $\delta_{\Lambda} = \delta_1$.
\end{proof}
We remark that Lemma 3.1 of \cite{mugnai2005exact} also gives a proof of \Cref{lem:invertibility eigengap} using contradiction. The proof we presented here shows the constant $c_{\Lambda}$ directly. We can now prove \Cref{lem:solve simplified 1}.

\begin{proof}[Proof of \Cref{lem:solve simplified 1}]
We only prove the lemma when $d \geq 3 $ and $p_k > \frac{2d}{d-2}$, because the proof for the other case in \Cref{def:space V} is the same.

First, $A_{\lambda}$ is by invertible according to \Cref{lem:invertibility eigengap}. We can then rewrite \eqref{eq:eigen rescaling 1 simplified} as
    \begin{align}
        \dphi_{\lambda} = - M (A_{\lambda})^{-1} \circ P^{\perp} \circ i^{\ast} \left[f\left(M^{-1}(w + \dphi_\lambda) \right) \right].
    \end{align}
Fix a constant $W>1$. For any fixed $w \in K_{\Lambda}$ with $\|w\|_V \leq W$, we define the map
    \begin{align}\label{eq:contraction mapping}
        T: V \cap K_{\Lambda} ^{\perp} \to V \cap K_{\Lambda} ^{\perp}, \ \phi \mapsto - M (A_{\lambda})^{-1} \circ P^{\perp} \circ i^{\ast} \left[f\left(M^{-1}(w + \phi) \right) \right].
    \end{align}
We show that there is a positive constant $\xi_{\Lambda}$ depending on $\Lambda,p_1, p_k,\Omega$, such that when $M > \xi_{\Lambda} W^{\frac{p_1 - 1}{p_1 - 2}}$, $T$ has a fixed point satisfying $\|\dphi_{\lambda}\|_V \leq M^{2-p_1} \xi_{\Lambda} ^{p_1 - 2} W^{p_1 - 1} $.

First, we compute $\| T(\phi)\|_V$ for a $\phi \in V \cap K_{\Lambda} ^{\perp}$ with $\|\phi\|_V \leq 1$. According to \eqref{eq:P,P perp norm} and \eqref{eq:invertibility eigengap V}, we have that
    \begin{align}\label{eq:T phi norm}
        \begin{split}
            \| T(\phi)\|_V &\leq M C \left\| i^{\ast} \left[f\left(M^{-1}(w + \phi) \right) \right] \right\|_V 
            \\  &=  M C \left(\left\| i^{\ast} \left[f\left(M^{-1}(w + \phi) \right) \right] \right\| + \left\| i^{\ast} \left[f\left(M^{-1}(w + \phi) \right) \right] \right\|_q \right)
            \\  &\leq  M C \left(\left\| f\left(M^{-1}(w + \phi) \right)  \right\|_{\frac{2d}{d+2}} + \left\| f\left(M^{-1}(w + \phi) \right) \right\|_\frac{dq}{d+2q} \right),
        \end{split}
    \end{align}
where we used \eqref{eq:i ast sobolev norm} and \Cref{lem:s norm inequality 1}.
In the above inequalities, the constant $C$ varies from line by line, but only depends on $\Lambda, p_1, p_k,  \Omega$. Because $q >2$, by H{\"o}lder's inequality, we only need to estimate the $L^{\frac{dq}{d+2q}}$-norm term above. We notice that, according to \Cref{rem:f inequality} and the choice of $q$ in \Cref{def:space V}, for any $v \in V$, 
    \begin{align}\label{eq:f(u) norm}
        \begin{split}
            \left\|f(v)\right\|_\frac{dq}{d+2q} &\leq C \left\|v\right\|_{\frac{(p_1 - 1)dq}{d+2q}} ^{p_1 - 1} + C' \left\|v\right\|_{\frac{(p_k - 1)dq}{d+2q}} ^{p_k - 1} 
            \leq C \left\|v\right\|_{\frac{(p_k - 1)dq}{d+2q}} ^{p_1 - 1} + C' \left\|v\right\|_{\frac{(p_k - 1)dq}{d+2q}} ^{p_k - 1}
            \\  & = C \left\|v\right\|_q ^{p_1 - 1} + C' \left\|v\right\|_q ^{p_k - 1} ,
        \end{split}
    \end{align}
where we used H{\"o}lder's inequality in the second inequality, and the constants $C,C'$ vary from line by line, but only depends on $\Lambda, p_1, p_k,  \Omega$. Hence, choose $M > 2W$ first and plug $v= M^{-1}(w + \phi)$ into \eqref{eq:f(u) norm}, we have that because $\|w + \phi\|_q \leq \|w + \phi\|_V \leq W +1 \leq 2W$, 
    \begin{align}
        \left\|f\left(M^{-1}(w + \phi) \right)\right\|_\frac{dq}{d+2q} \leq (C+C') M^{1-p_1} (2W)^{p_1 - 1} .
    \end{align}
Thus, for some constant $C$ depending on $\Lambda, p_1, p_k,  \Omega$,
    \begin{align}
        \| T(\phi)\|_V \leq C M^{2-p_1} W^{p_1 - 1}.
    \end{align}
In particular, when $M > C^{\frac{1}{p_1 - 2} } W ^{\frac{p_1 - 1}{p_1 - 2}}$, we have that $\| T(\phi)\|_V < 1$.

Next, we show that $T$ is a contraction mapping. Similar to \eqref{eq:T phi norm}, for any $\phi_1, \phi_2 \in V \cap K_{\Lambda} ^{\perp}$ with $\|\phi_1\|_V, \|\phi_2\|_V \leq 1$, we have that
    \begin{align}
        \begin{split}
            \| T(\phi_1) - T(\phi_2)\|_V &\leq M C  \left\| f\left(M^{-1}(w + \phi_1) \right) - f\left(M^{-1}(w + \phi_2) \right) \right\|_\frac{dq}{d+2q} 
            \\  &=C  \left\| f ' \left(M^{-1}(w + (\theta \phi_1+(1-\theta)\phi_2) \right)(\phi_1 - \phi_2) \right\|_\frac{dq}{d+2q} ,
        \end{split}
    \end{align}
where we used the mean value theorem and $\theta \in (0,1)$. 
Similar to \eqref{eq:f(u) norm}, we see that for any $v \in V$ and any $r > 1$,
    \begin{align}\label{eq:f'(u) norm}
        \begin{split}
            \left\|f'(v)\right\|_r &\leq C \left\|v\right\|_{r(p_1-2)} ^{p_1 - 2} + C' \left\|v\right\|_{r(p_k-2)} ^{p_k - 2} 
            \leq  C \left\|v\right\|_{r(p_k-2)} ^{p_1 - 2} + C' \left\|v\right\|_{r(p_k-2)} ^{p_k - 2} ,
        \end{split}
    \end{align}
where the constants $C,C'$ vary from line by line, but only depends on $\Lambda, p_1, p_k,  \Omega$. Hence, choose $M > 2W$ and plug $v= M^{-1}(w + (\theta \phi_1+(1-\theta)\phi_2)$ into \eqref{eq:f(u) norm}, we have that because $\|w + (\theta \phi_1+(1-\theta)\phi_2)\|_q \leq \|w + (\theta \phi_1+(1-\theta)\phi_2)\|_V \leq W+1\leq 2W$, and because $\frac{(p_k - 1) d q}{d + 2q} = q$,
    \begin{align}
        \begin{split}
            \| T(\phi_1) - T(\phi_2)\|_V &\leq  C  \left\| f ' \left(M^{-1}(w + (\theta \phi_1+(1-\theta)\phi_2) \right) \right\|_{\frac{(p_k - 1)dq}{(p_k - 2)(d+2q)}} \left\|(\phi_1 - \phi_2) \right\|_{\frac{(p_k - 1)dq}{d+2q}}
            \\  &\leq C M^{2-p_1} W^{p_1 - 2} \left\|(\phi_1 - \phi_2) \right\|_{q} \leq C M^{2-p_1} W^{p_1 - 2} \left\|(\phi_1 - \phi_2) \right\|_{V}. 
        \end{split}
    \end{align}
Hence, when $M > C^{\frac{1}{p_1 - 2} } W$, $T$ is a contraction mapping. The existence of a unique fixed point $\dphi_\lambda \in (K_{\Lambda} ^{\perp} \cap V) $ with $\|\dphi_\lambda\|_V < 1$ of $T$ then follows from the contraction mapping theorem. This $\dphi_\lambda$ then solves \eqref{eq:eigen rescaling 1 simplified}. The $C^1$-differentiability with respect to $w \in K_{\Lambda}$ comes from the differentiability of $T$ with respect to $w \in K_{\Lambda}$ as defined in \eqref{eq:contraction mapping}.
\end{proof}

%-------------------------------------------------

\subsection{Proof of \Cref{lem:solve simplified 2}}

We reduce solving \eqref{eq:eigen rescaling 2 simplified} to finding a critical point of an energy functional defined on the finite dimensional space $K_{\Lambda}$. First, we notice that solutions to \eqref{eq:eigen rescaling} is equivalent to critical points of the energy functional
    \begin{align}
        \widetilde E_{\Omega}(\du) \coloneqq \int_{\Omega} \frac{1}{2}|\nabla \du(x)|^2 - \frac{\lambda}{2} \du^2 + M^2 F(M^{-1}\du(x)) \ \de x, \quad \du \in H_0 ^1 (\Omega).
    \end{align}
We also define an energy functional on $K_{\Lambda}$. For any $w \in K_{\Lambda}$, we denote $\dphi_{\lambda}(w)$ as the solution to \eqref{eq:eigen rescaling 1 simplified} we obtained in \Cref{lem:solve simplified 1} after assuming $M > C_{\Lambda} \|w\|_V ^{\frac{p_1 - 1}{p_1-2}}$, and we define
    \begin{align}
        \widetilde J(w) \coloneqq \widetilde E_{\Omega}(w + \dphi_{\lambda}(w)).
    \end{align}
\begin{lemma}\label{lem:equivalent critical point}
    Adopt the assumptions of \Cref{lem:solve simplified 1}. A function $\du = w + \dphi_\lambda (w)$ is a solution to \eqref{eq:eigen rescaling} if and only if $w$ is a critical point of $\widetilde J$ on $K_{\Lambda}$.
\end{lemma}
\begin{proof}
    By construction in \Cref{lem:solve simplified 1}, for any $w \in K_\Lambda$, the function $\dphi_\lambda(w)$ satisfies (2.6), which ensures that $\du = w + \dphi_\lambda(w)$ satisfies the projected equation onto $K_\Lambda^\perp$. This is equivalent to the condition that $\langle \de \widetilde{E}_\Omega(\du), \xi \rangle = 0$ for all $\xi \in K_\Lambda^\perp $. Here, $\de \widetilde{E}_\Omega(\du)$ is the Fr{\'e}chet derivative of $\widetilde{E}_\Omega$ at $\du$ in $H_0 ^1 (\Omega)$.
Now, compute the derivative of $\widetilde{J}$ at $w$:
    \begin{align}
        \de\widetilde{J}(w)[v] = \frac{\de}{\de t}\bigg|_{t=0} \widetilde{J}(w + tv) = \frac{\de}{\de t}\bigg|_{t=0} \widetilde{E}_\Omega((w + tv) + \widetilde{\phi}_\lambda(w + tv)) = \langle \de \widetilde{E}_\Omega(\du), v + \psi \rangle,
    \end{align}
where $\psi = \de \dphi_\lambda(w)[v]$ and $v \in K_\Lambda$. From \Cref{lem:solve simplified 1}, $\dphi_\lambda(w) \in K_\Lambda^\perp \cap V$ and is $C^1$-differentiable with respect to $w$, so $\psi$ is well-defined and $\psi \in K_\Lambda^\perp \cap V$. Thus, $\langle \de \widetilde{E}_\Omega(\du), \psi \rangle = 0$, and
    \begin{align}
        \de\widetilde{J}(w)[v] = \langle \de \widetilde{E}_\Omega(\du), v  \rangle .
    \end{align}
Therefore, $w$ is a critical point of $\widetilde J$ on $K_{\Lambda}$ if and only if $\de\widetilde{J}(w)[v] = 0$ for all $v \in K_\Lambda$, and if and only if $\langle \de \widetilde{E}_\Omega(\du), v  \rangle = 0$ for all $v \in K_\Lambda$. Since $\langle \de \widetilde{E}_\Omega(\du), \xi  \rangle$ already holds for all $\xi \in K_\Lambda^\perp$ by construction, this is equivalent to $\de \widetilde{E}_\Omega(\du) = 0$, meaning $\du = w + \dphi_\lambda(w)$ is a critical point of $\widetilde{E}_\Omega$ and thus a solution to \eqref{eq:eigen rescaling}.
\end{proof}

For functions $w \in K_{\Lambda}$, we have the following inverse H\"older inequalities.
\begin{lemma}\label{lem:sogges estimate}
    Let $M$ be a manifold (with or without boundary) of dimension $N$, and $f$ is a Laplacian eigenfunction satisfying $\Delta_{M}f + \Lambda f = 0$. Then, there is a constant $C$ only depending on $M$, such that for any $r \geq 2$,
        \begin{align}
            \|f\|_r \leq C \Lambda^{\frac{N}{4}} \|f\|_2.
        \end{align}
\end{lemma}
For a proof for \Cref{lem:sogges estimate}, see for example the equation (1.8) and (1.9) in \cite{smith2007p} and \cite{sogge1988concerning}. We remark that the sharp order is smaller than $\Lambda^{\frac{N}{4}}$, as shown in \cite{smith2007p,sogge1988concerning}, but the order $\Lambda^{\frac{N}{4}}$ is enough for the purpose of our paper.

Next, because $w \in K_{\Lambda}$ and $\dphi_{\lambda}(w) \in K_{\Lambda} ^{\perp}$, they are orthogonal to each other in both $H_0 ^1 (\Omega)$ and $L^2 (\Omega)$. Hence, we can expand $\widetilde{J}(w)$ as
    \begin{align}
        \begin{split}
            \widetilde{J}(w) &= \int_{\Omega} \frac{1}{2}|\nabla w|^2 + \frac{1}{2} |\nabla \dphi_{\lambda}(w)|^2 - \frac{\lambda}{2} w^2- \frac{\lambda}{2} (\dphi_{\lambda}(w)) ^2 + M^2 F(M^{-1}(w + \dphi_{\lambda}(w))) \ \de x
            \\  & = \widetilde{G}(w) + \widetilde{R}(w), 
        \end{split}
    \end{align}
where
    \begin{align}
        \widetilde{G}(w) \coloneqq \frac{\Lambda - \lambda}{2} \| w \|_2 ^2 + M^2 \int_{\Omega} F(M^{-1}w) \ \de x,
    \end{align}
and 
    \begin{align}
        \widetilde{R}(w) \coloneqq \int_{\Omega} \frac{1}{2} |\nabla \dphi_{\lambda}(w)|^2 - \frac{\lambda}{2} (\dphi_{\lambda}(w)) ^2 + M^2 \left[ F(M^{-1}(w + \dphi_{\lambda}(w)))  - F(M^{-1}w)\right]\ \de x.
    \end{align}

\begin{lemma}\label{lem:critical point for tilde G}
    There exist a positive constant $\eta<1$ depending on $b_0,\Omega,p_1$, a constant $\varepsilon_{\Lambda}<\frac{1}{2}$ of the form $\varepsilon_{\Lambda} = C\Lambda^{-\kappa} \eta^{\frac{1}{p_1 - 2}}$ where $C>0$ depends on $\Omega,p_1,p_k,b,b_0$ and $\kappa>0$ depends on $d,p_1,p_k$, such that if we let $\lambda = \Lambda + 2 \eta M^{2-p_1}$, we have that $\widetilde{G}(w) \geq -(\frac{\eta}{4} \varepsilon_{\Lambda} ^2 )M^{2-p_1}$ when $\|w\|_2 \in [0,\frac{\varepsilon_{\Lambda}}{2}]$, $\widetilde{G}(w) \leq -(\frac{\eta}{2} \varepsilon_{\Lambda} ^2 )M^{2-p_1}<0$ when $\|w\|_2 = \varepsilon_{\Lambda}$, and $\widetilde{G}(w) \geq (\frac{\eta}{4})M^{2-p_1} >0$ when $\|w\|_2 = \frac{1}{2}$. 
\end{lemma}
\begin{proof}
    By the choice of $\lambda$, we see that 
        \begin{align}
            \widetilde{G}(w) = M^{2-p_1} \left( -\eta \|w \|_2 ^2 + M^{p_1} \int_{\Omega} F(M^{-1}w) \ \de x \right).
        \end{align}
    By \Cref{lem:sogges estimate}, there is a constant $C$ depending on $\Omega$, such that for any $r \geq 2$ and any $w \in K_{\Lambda}$,  
        \begin{align}\label{eq:equivalent norm}
            \|w\|_r \leq C \Lambda^d \|w\|_2 .
        \end{align}
    Hence, when $\|w\|_2 \leq 1 $ and $M >1$, we know that 
        \begin{align}
            \begin{split}
                M^{p_1} \int_{\Omega} F(M^{-1}w) \ \de x &\leq C_1  M^{p_1} \left( M^{-p_1} \|w\|_{p_1} ^{p_1} + M^{-p_2} \|w\|_{p_2} ^{p_2} + \dots + M^{-p_k} \|w\|_{p_k} ^{p_k}\right)
                \\  &\leq  C_1  M^{p_1} \Lambda^{d p_k} \left( M^{-p_1} \|w\|_{2} ^{p_1} + M^{-p_2} \|w\|_{2} ^{p_2} + \dots + M^{-p_k} \|w\|_{2} ^{p_k}\right)
                \\  &\leq C_1  M^{p_1} \Lambda^{d p_k}\left( M^{-p_1} \|w\|_{2} ^{p_1}\right) = C_1 \Lambda^{d p_k} \|w\|_{2} ^{p_1},
            \end{split}
        \end{align}
where the constant $C_1>1$ varies from line by line but only depend on $\Omega,p_1,p_k,b$. Also, we have that, by \Cref{a:superlinearity},
        \begin{align}
            M^{p_1} \int_{\Omega} F(M^{-1}w) \ \de x \geq b_0  \int_{\Omega} |w|^{p_1} \ \de x \geq C_2 \|w\|_2 ^{p_1},
        \end{align}
for a constant $C_2<1$ depending on $b_0,\Omega$ by the H\"{o}lder inequality. Combine these two inequalities, we see that 
    \begin{align}\label{eq:order tilde G}
         M^{2-p_1} \|w\|_2 ^2 \left( -\eta + C_2 \|w\|_2 ^{p_1 - 2} \right) \leq \widetilde{G}(w) \leq M^{2-p_1} \|w\|_2 ^2 \left( -\eta + C_1 \Lambda^{d p_k} \|w\|_2 ^{p_1 - 2} \right).
    \end{align}

We now pick 
    \begin{align}
        \eta = \frac{C_2}{2^{p_1 - 1}} < 1.
    \end{align}
When $\|w\|_2 = \frac{1}{2}$, $ -\eta + C_2 \|w\|_2 ^{p_1 - 2} = \eta$, and hence $\widetilde{G}(w)  \geq (\frac{\eta}{4})M^{2-p_1}>0$. Next, we choose
    \begin{align}
        \varepsilon_{\Lambda} =\left( \frac{\eta}{2 C_1 \Lambda^{d p_k}} \right)^{1/(p_1-2)}.
    \end{align}
When $\|w\|_2 = \varepsilon_{\Lambda}$, $-\eta + C_1 \Lambda^{d p_k}\|w\|_2 ^{p_1 - 2} \leq  -\eta+ \frac{\eta}{2}  = -\frac{\eta}{2}$, and hence $\widetilde{G}(w)  \leq  -(\frac{\eta}{2} \varepsilon_{\Lambda} ^2 )M^{2-p_1}<0$. Finally, when $\|w\|_2 \in [0,\frac{\varepsilon_{\Lambda}}{2}]$, \eqref{eq:order tilde G} indicates that $\widetilde{G}(w) \geq- \eta  \|w\|_2 ^2 M^{2-p_1} \geq -(\frac{\eta}{4} \varepsilon_{\Lambda} ^2 )M^{2-p_1}$.
\end{proof}

\begin{lemma}\label{lem:super small tilde R}
    Adopt the assumptions of \Cref{lem:solve simplified 1}. There exists a constant $C$ depending on $\Omega,p_1,p_k,b$, such that for any $w \in K_{\Lambda}$ with $\|w\|_2 \leq 1$,
        \begin{align}
            |\widetilde{R}(w) |\leq C\Lambda^{dp_k}\left( \|\dphi_{\lambda}(w)\|_V ^2 + M^{2-p_1} \|\dphi_{\lambda}(w)\|_V \right).
        \end{align}
\end{lemma}
\begin{proof}
    First, because $|\lambda - \Lambda| \leq 1$, in $\widetilde{R}(w)$, 
        \begin{align}
            \left| \int_{\Omega} \frac{1}{2} |\nabla \dphi_{\lambda}(w)|^2 - \frac{\lambda}{2} (\dphi_{\lambda}(w)) ^2 \de x \right| \leq C \Lambda \|\dphi_{\lambda}(w)\|_V ^2 ,
        \end{align}
    by \Cref{def:space V} and the Poincar{\'e} inequality, where $C$ only depends on $\Omega$.

    Next, we consider the nonlinear term 
        \begin{align}
             M^2 \int_{\Omega}  \left[ F(M^{-1}(w + \dphi_{\lambda}(w)))  - F(M^{-1}w)\right]\ \de x.
        \end{align}
    By the mean value theorem, for some $\theta \in (0,1)$ and $\zeta = w + \theta \dphi_{\lambda}(w)$, and we have that
    \begin{align}
        \begin{split}
            &\left| F(M^{-1}(w + \dphi_{\lambda}(w)))  - F(M^{-1}w)\right| \leq \left|f(M^{-1}\zeta) \right| \left|M^{-1} \dphi_{\lambda}(w) \right| 
            \\  &\leq C\left[ |M^{-1} \zeta|^{p_1 - 1} +  | M^{-1} \zeta|^{p_k - 1} \right] M^{-1} |\dphi_{\lambda}(w)|
            \\  &\leq C\left[ |w|^{p_1 - 1} + |w|^{p_k - 1}+ |\dphi_{\lambda}(w)|^{p_1 - 1} +  |\dphi_{\lambda}(w)|^{p_k - 1} \right] M^{-p_1} |\dphi_{\lambda}(w)|,
        \end{split}
    \end{align}
using \Cref{rem:f inequality}, where $C$ depends on $b, p_1, p_k$. We estimate the above terms separately. By H{\"o}lder inequality and \eqref{eq:equivalent norm},
    \begin{align}
        \int_{\Omega} |w|^{p_1 - 1} |\dphi_{\lambda}(w)| \de x \leq \|w\|_{\frac{(p_1 - 1)q}{q-1}} ^{p_1 - 1} \|\dphi_{\lambda}(w)\|_q \leq C \Lambda^{dp_1} \|w\|_{2} ^{p_1 - 1} \|\dphi_{\lambda}(w)\|_q =  C \Lambda^{dp_1} \|\dphi_{\lambda}(w)\|_q,
    \end{align}
where $C$ only depends on $\Omega$.
We can similarly obtain $\int_{\Omega} |w|^{p_k - 1} |\dphi_{\lambda}(w)| \de x \leq C \Lambda^{dp_k}\|\dphi_{\lambda}(w)\|_q$ for some $C$ only depending on $\Omega$. For the other two terms, because $p_1 < p_k < q$, we see that, by H{\"o}lder inequality, $\|\dphi_{\lambda}(w)\|_{p_1} ^{p_1} \leq C \|\dphi_{\lambda}(w)\|_{q} ^{p_1} \leq C \|\dphi_{\lambda}(w)\|_{q}$, because $\|\dphi_{\lambda}(w)\|_{V} < 1$ as in \Cref{lem:solve simplified 1}. Similarly, $\|\dphi_{\lambda}(w)\|_{p_k} ^{p_k} \leq C \|\dphi_{\lambda}(w)\|_{q}$. Combine the above estimates, we have that 
    \begin{align}
        M^2 \int_{\Omega}  \left[ F(M^{-1}(w + \dphi_{\lambda}(w)))  - F(M^{-1}w)\right]\ \de x \leq C M^{2-p_1} \Lambda^{dp_k}\|\dphi_{\lambda}(w)\|_{q} \leq C M^{2-p_1} \Lambda^{dp_k}\|\dphi_{\lambda}(w)\|_{V},
    \end{align}
where $C$ depends on $b, p_1, p_k,\Omega$.

\end{proof}

\begin{proof}[Proof of \Cref{lem:solve simplified 2}]
First, by \eqref{eq:equivalent norm} and the fact that $\|w\|^2 = \Lambda\|w\|_2 ^2$, there is a constant $C$ depending on $\Omega$, such that $\|w\|_V \leq C \Lambda^d \|w\|_2$ for any $w \in K_{\Lambda}$. Hence, when $\|w\|_2 \leq 1$,  $\|w\|_V \leq C \Lambda^d$, and we choose the parameter $W=C \Lambda^d$ in \Cref{lem:solve simplified 1}. By \Cref{lem:solve simplified 1}, we know that $\|\dphi_{\lambda}(w)\|_V \leq C' M^{2-p_1} $ for some $C'$ depending on $\Lambda,\Omega,p_1,p_k,b,b_0$ when $M>(C')^{\frac{1}{p_1 - 2}}$. Hence, by \Cref{lem:super small tilde R}, we see that $|\widetilde{R}(w) | \leq C \Lambda^{dp_k} M^{4-2p_1}$ for some $C$ depending on $\Omega,p_1,p_k,b$. By \Cref{lem:critical point for tilde G}, we see that if we further assume that $C\Lambda^{dp_k} M^{2-p_1} < \frac{\eta}{8} \varepsilon_{\Lambda} ^2 $, we have that $|\widetilde{R}(w) | < (\frac{\eta}{8} \varepsilon_{\Lambda} ^2) M^{2-p_1} $. Hence, we have that when $\|w\|_2 \in [0,\frac{\varepsilon_{\Lambda}}{2}]$,
        \begin{align}
            \widetilde{J}(w) \geq \widetilde{G}(w) - |\widetilde{R}(w)| > \widetilde{G}(w) - \left(\frac{\eta}{8} \varepsilon_{\Lambda} ^2 \right)M^{2-p_1} \geq  - \left(\frac{3\eta}{8} \varepsilon_{\Lambda} ^2 \right)M^{2-p_1},
        \end{align}
        when $\|w\|_2 = \varepsilon_{\Lambda}$, 
        \begin{align}
            \widetilde{J}(w) \leq \widetilde{G}(w) + |\widetilde{R}(w)| < \widetilde{G}(w) + \left(\frac{\eta}{8} \varepsilon_{\Lambda} ^2 \right)M^{2-p_1} \leq -\left(\frac{3\eta}{8} \varepsilon_{\Lambda} ^2 \right) M^{2-p_1}<0,
        \end{align}
    and when $\|w\|_2 = \frac{1}{2}$,
        \begin{align}
            \widetilde{J}(w) \geq \widetilde{G}(w) - |\widetilde{R}(w)| \geq \widetilde{G}(w) - \left(\frac{\eta}{8}\right)M^{2-p_1} \geq \left(\frac{\eta}{8}  \right)M^{2-p_1} > 0.
        \end{align}
The above three inequalities indicate that
    \begin{align}
        \min_{\|w\|_2 \in [0,\frac{\varepsilon_{\Lambda}}{2}]} \widetilde{J}(w) > \min_{\|w\|_2 = \varepsilon_{\Lambda}} \widetilde{J}(w), \text{ and }  \min_{\|w\|_2 = \frac{1}{2}} \widetilde{J}(w) > \min_{\|w\|_2 = \varepsilon_{\Lambda}} \widetilde{J}(w).
    \end{align}
Hence,
in the domain of $w \in K_{\Lambda}$ with $\|w\|_2 \leq \frac{1}{2}$, $\widetilde{J}(w)$ must reaches its minimum value at a $\dw_{\lambda} \in K_{\Lambda}$ satisfying $\|\dw_{\lambda}\|_2 \in (\frac{\varepsilon_{\Lambda}}{2} , \frac{1}{2})$. Hence, this $\dw_{\lambda}$ is a critical point of the energy functional $\widetilde{J}(w)$. By \Cref{lem:equivalent critical point}, the function $\du = \dw_\lambda + \dphi_\lambda (\dw_{\lambda})$ is a solution to \eqref{eq:eigen rescaling}.
\end{proof}

\subsection{Proof of \Cref{thm:main bifurcation}}
From \Cref{lem:solve simplified 1} and \Cref{lem:solve simplified 2}, for each $M>C'_\Lambda$ sufficiently large (depending on $\Lambda,\Omega,p_1,p_k,b,b_0$), we obtained $\lambda=\Lambda+2\eta M^{2-p_1}$ with $|\lambda-\Lambda|=2\eta M^{2-p_1}<\delta_\Lambda$, a function $\dw_\lambda\in K_\Lambda$ with $\|\dw_\lambda\|_2^2\in(\varepsilon_\Lambda/2,1/2)$, and $\dphi_\lambda = \dphi_\lambda(\dw_\lambda)\in K_\Lambda^\perp\cap V$ with $\|\dphi_\lambda\|_V\leq M^{2-p_1}(C'_\Lambda)^{p_1-2}$, such that $\du=\dw_\lambda+\dphi_\lambda$ solves the rescaled equation \eqref{eq:eigen rescaling},
and $\dphi_\lambda\perp\dw_\lambda$ in $H_0^1(\Omega)$. Setting $u_\lambda=\du/M$ yields a solution to \eqref{eq:eigen}.
Furthermore, we see that
\begin{align}
    M=\left(\frac{2\eta}{\lambda - \Lambda}\right)^{\frac{1}{p_1 - 2}}.
\end{align}
Set 
    \begin{align}
        w_\lambda=\dw_{\lambda} \cdot (2\eta)^{-\frac{1}{p_1 - 2}}, \quad \phi_\lambda=\dphi_{\lambda} \cdot (2\eta)^{-\frac{1}{p_1 - 2}},
    \end{align}
we see that $u_{\lambda} = (\lambda - \Lambda)^{\frac{1}{p_1 - 2}} \left( w_{\lambda}+ \phi_{\lambda}\right)$.
We can then obtain those constants $\alpha,\beta ,\kappa, \gamma_{\Lambda}$ from the conclusion of \Cref{lem:solve simplified 2}.

%----------------------------------------
\section{Proof of \Cref{thm:main whispering gallery linear eigen}}\label{sec:proof of main whispering gallery linear eigen}

First, we recall some basic facts of Dirichlet Laplacian eigenfunctions of the unit ball $\Omega = B_1 \subseteq \bR^d$ ($d \geq 2$). See, for example, $\S 5$ in Chapter V of \cite{CH89}, or \cite{NG2013,kuperman2026clamped}. We define notations $D \coloneqq \frac{d}{2} - 1 $ and $\nu \coloneqq n + D$. For any $\mu \geq 0$, we use $J_\mu(s)$ to denote the Bessel function of the first kind of order $\mu$ and use $J_{\mu} '(s)$ to denote its derivative. For any $\mu \geq 0$, denote $j_{\mu}$ as the first positive root of $J_{\mu}(s)$ and denote  $j_{\mu} '$ as the first positive root of $J_{\mu} ' (s)$. That is,
\begin{align}\label{def:j_mu}
        j_{\mu} \coloneqq \sup \{ t \ | \ J_{\mu} (s) \neq 0 \text{ for any } s \in (0,t), \ t>0 \},
\end{align}
and we similarly define $j_{\mu} '$.

Then the following functions form a basis for Dirichlet Laplacian eigenfunctions with eigenvalue $\Lambda_n \coloneqq j_{\nu} ^2$ on $\Omega = B_1$:
    \begin{align}\label{e:sharp eigenfunction}
        w_{n,m}(x) \coloneqq r^{-D} J_{\nu}(j_{\nu} r) Y_{n,m} (\theta), \ 0<r \leq 1, \ \theta \in \bS ^{d-1}, \ n \in \bZ_{\geq 0}, \ 1 \leq m \leq N_n,
    \end{align}
with eigenvalue $\Lambda_n $, i.e., $\Delta w_{n,m}+ \Lambda_n w_{n,m} = 0$, where we used the spherical coordinates for $x =(r,\theta) \in \R^d$, and $Y_{n,m} (\theta)$ is a spherical harmonic of degree $n$, i.e., $\Delta_{\bS^{d-1}} Y_{n,m} + n(n+d-2) Y_{n,m} = 0$, and satisfies that $\int_{\bS^{d-1}} Y_{n,m} ^2(\theta) \de \theta = 1$. Here, $\{Y_{n,m}\}_{m=1} ^{N_n}$ form an orthonormal basis for spherical harmonics of degree $n$ on the sphere $\bS^{d-1}$, and $N_n = \binom{n+d-1}{d-1} - \binom{n+d-3}{d-1}$ is the dimension of this linear space. $\{w_{n,m}\}_{m=1} ^{N_n}$ then form an orthogonal basis for Dirichlet eigenfunctions with the eigenvalue $\Lambda_n$ of the unit ball $\Omega=B_1$, and hence $K_{\Lambda_n} = \mathrm{Span} \{w_{n,m}\}_{m=1} ^{N_n}$. For any $w \in K_{\Lambda_n}$, there is a unique $a = (a_1,a_2,\dots,a_{N_n} ) \in \R^{N_n}$, such that 
    \begin{align}
        w = \sum_{m=1} ^{N_n} a_m w_{n,m} = r^{-D} J_{\nu}(j_{\nu} r) \sum_{m=1} ^{N_n} a_m Y_{n,m} (\theta) .
    \end{align}
We denote $Y(\theta) \coloneqq  \sum_{m=1} ^{N_n} a_m Y_{n,m} (\theta)$, and thus $w = r^{-D} J_{\nu}(j_{\nu} r) Y(\theta)$.
To simplify the notations, in the following lemmas and proofs, we assume that $w \in K_{\Lambda_n}$ and satisfies $\|a\|_2 ^2 = \sum_{m=1} ^{N_n} a_m ^2 = 1$. Hence, $\|Y\|_{L^2(\bS^{d-1})} ^2 = \sum_{m=1} ^{N_n} a_m ^2 = 1$.

In the following proofs, when we say that $\nu$ is large enough, we mean that $\nu$ is larger than some large universal constant.

\begin{lemma}\label{lem:gradient L2 norm}
    For any $s \geq 0$,
        \begin{align}\label{eq:gradient L2 norm}
            \begin{split}
                \int_{B_s} |\nabla w(x)|^2 \de x &= \frac{j_{\nu} ^2 s^2 - \nu^2}{2} J_{\nu} ^2 (j_{\nu} s) + \frac{j_{\nu} ^ 2 s^2 }{2} \left( J_{\nu} ' (j_{\nu} s) \right) ^2 
                \\  &+ \left(-D J_{\nu}(j_{\nu} s) +  j_{\nu} s J_{\nu} ' (j_{\nu} s)\right) J_{\nu}(j_{\nu} s) .
            \end{split}
        \end{align}
\end{lemma}

\begin{proof}[Proof of \Cref{lem:gradient L2 norm}]
Denote $R(r)=r^{-D}J_{\nu}(j_{\nu} r)$, and then we see that $w = R(r) Y(\theta)$. By integration by parts,
    \begin{align}
        \int_{B_s}|\nabla w|^2 \de x=-\int_{B_s}w\Delta w \de x+\int_{\partial B_s}w\partial_r w \de \sigma(x).
    \end{align}
The first term is $j_{\nu} ^2\int_{B_s}w^2 \de x$. Moreover, $\int_{B_s}w^2 \de x=\int_0^s r^{d-1}R(r)^2\de r=\int_0^s r J_{\nu}(j_{\nu} r)^2\de r$, because $d-1-2D=1$. The boundary term is
\begin{align}
    \int_{\bS^{d-1}}[R(s)Y_{n,m}(\theta)][R'(s)Y_{n,m}(\theta)]s^{d-1}\de \theta=s^{d-1}R(s)R'(s) \int_{\bS^{d-1}} Y_{n,m} ^2(\theta) \de \theta=s^{d-1}R(s)R'(s).
\end{align}
Now, because $R'(s)=-Ds^{-D-1}J_{\nu}(\lambda s)+s^{-D}\lambda J_{\nu}'(\lambda s)$, the boundary term becomes
    \begin{align}
        s^{d-1}R(s)R'(s)=-DJ_{\nu}(\lambda s)^2+\lambda s J_{\nu}(\lambda s)J_{\nu}'(\lambda s),
    \end{align}
since $d-1-2D=1$. Hence,
    \begin{align}
        \int_{B_s}|\nabla w|^2\,dx=j_{\nu} ^2\int_0^sr J_{\nu}(j_{\nu} r)^2 \de r+j_{\nu} s J_{\nu}(j_{\nu} s)J_{\nu}'(j_{\nu} s)-D J_{\nu} ^2 (j_{\nu} s).
    \end{align}
By Lommel's integral (see for example  Equation (11) in Section 5.11 of Chapter 5 in \cite{W1995} or Lemma 5.4 of \cite{kuperman2026clamped}), 
    \begin{align}
         j_{\nu} ^2 \int_0^s r J_{\nu}(\lambda r)^2 \de r = \int_0 ^{j_{\nu} s} r J_{\nu}(r)^2 \de r = \frac{j_{\nu} ^2 s^2 - (\nu)^2}{2} J_{\nu} ^2 (j_{\nu} s) + \frac{j_{\nu} ^ 2 s^2 }{2} \left( J_{\nu} ' (j_{\nu} s) \right) ^2 .
    \end{align}
This finishes the proof for \Cref{lem:gradient L2 norm}.
\end{proof}

Next, we need the following explicit asymptotics related to $j_{\mu}$ and $j_{\mu} '$. See equation (9.5.14), equation (9.5.16), equation (9.5.18) in Section 9.5 of \cite{AS1964} or  Appendix A of \cite{NG2013} or equation (30) of \cite{keller1960asymptotic} for the following \Cref{lem:j_mu asymptotics}.
\begin{lemma}[Section 9.5 of \cite{AS1964}]\label{lem:j_mu asymptotics}
    For any $\mu \geq 0$ large enough,
        \begin{align}\label{e:j_mu asymptotics}
            \mu + (1.855) \cdot \mu^{\frac{1}{3}}< j_\mu < \mu + (1.856) \cdot \mu^{\frac{1}{3}},
        \end{align}
    and 
        \begin{align}\label{e:j_mu ' asymptotics}
            \mu + (0.808) \cdot \mu^{\frac{1}{3}}< j_\mu ' < \mu + (0.809) \cdot \mu^{\frac{1}{3}},
        \end{align}
and
\begin{align}\label{eq:J_mu ' at zero asymptotics}
            (-1.114) \cdot \mu^{-\frac{2}{3}} < J_\mu ' (j_{\mu}) < (-1.113 ) \cdot \mu^{-\frac{2}{3}},
\end{align}
and
\begin{align}\label{eq:J_mu at max asymptotics}
            (0.674) \cdot \mu^{-\frac{1}{3}} < J_\mu (j_{\mu} ') < (0.675) \cdot \mu^{-\frac{1}{3}}.
\end{align}
\end{lemma}

\begin{lemma}\label{lem:supercritical region}
    For any $\mu \geq 0$ large enough,
        \begin{align}\label{eq:supercritical region j_mu}
            0<J_{\mu}(\mu-\mu^{\frac{2}{3} }) < 2^{-\frac{1}{3}\mu ^{\frac{1}{3}}},
        \end{align}
    and 
        \begin{align}\label{eq:supercritical region j_mu '}
            \left|J_{\mu} '(\mu-\mu^{\frac{2}{3} }) \right|< \mu^{-\frac{1}{2}} 2^{-\frac{1}{3}\mu ^{\frac{1}{3}}}.
        \end{align}
\end{lemma}
\begin{proof}
    \eqref{eq:supercritical region j_mu} is proved in Lemma A.1 of \cite{NG2013} or \cite{kroger1996ground}, so we omit its proof. For \eqref{eq:supercritical region j_mu '}, according to the equaiton (10) in Section 8.5 of Chapter 8 in \cite{W1995} (or \cite{siegel1953inequality} and Section 10.14 of \cite{NIST:DLMF}), we have that for any $0 < x \leq 1$,
        \begin{align}
            \left| J_{\mu}'(\mu x) \right| \leq  \frac{(1+x^2)^{\frac{1}{4}} x^{\mu} e^{\mu \sqrt{1-x^2}}}{x \sqrt{(2\pi \mu)}(1+\sqrt{1-x^2})^{\mu}}.
        \end{align}
    The proof for \eqref{eq:supercritical region j_mu} in \cite{kroger1996ground} shows that for any $\varepsilon \in (0,\frac{2}{3})$ and any $x \in (0,1-\mu^{\varepsilon-\frac{2}{3}})$,
        \begin{align}
             \frac{ x^{\mu} e^{\mu \sqrt{1-x^2}}}{(1+\sqrt{1-x^2})^{\mu}} < 2^{-\frac{\mu^\varepsilon}{3}}.
        \end{align}
    In particular, choosing $\varepsilon = \frac{1}{3}$ and $x = 1-\mu^{-\frac{1}{3}}$, we see that the right hand side of the above inequality is $2^{-\frac{1}{3}\mu ^{\frac{1}{3}}}$. Also, for $\mu$ large enough, $x$ is close to $1$ and hence $\frac{(1+x^2)^{\frac{1}{4}} }{x \sqrt{(2\pi \mu)}} < \mu^{-\frac{1}{2}}$. This proves \eqref{eq:supercritical region j_mu '}.
\end{proof}

\begin{lemma}\label{lem:gradient L2 whispering}
    When $\nu$ is large enough, for any $m \in \llbracket 1 , N_n \rrbracket$,
    \begin{align}
        \frac{1}{2} \nu^{-\frac{4}{3}} \leq \int_{B_{1}} |w(x)|^2 \de x = j_{\nu} ^{-2} \int_{B_{1}} |\nabla w(x)|^2 \de x \leq \nu^{-\frac{4}{3}}.
    \end{align}
    Also, letting $\zeta_n = j_{\nu} ^{-1} (\nu - \nu ^{\frac{2}{3}}) < 1$, we have that when $\nu$ large enough,
    \begin{align}
        \int_{B_{\zeta_n}} |\nabla w(x)|^2 \de x < 2^{-\frac{1}{3}\nu ^{\frac{1}{3}}} .
    \end{align}

\end{lemma}
\begin{proof}
    From \Cref{lem:gradient L2 norm}, because $ J_{\nu} (j_\nu)=0$, we first see that
        \begin{align}
            \int_{B_{1}} |\nabla w(x)|^2 \de x = \frac{j_{\nu} ^ 2 }{2} \left( J_{\nu} ' (j_{\nu}) \right) ^2 .
        \end{align}
    By \Cref{lem:j_mu asymptotics}, we see that $\nu ^{-\frac{4}{3}} \leq (J_{\nu} ' (j_{\nu}))  ^2 \leq 2 \nu ^{-\frac{4}{3}}$ and $j_{\nu} > \nu$. Because $w$ is a Laplacian eigenfunction of eigenvalue $\Lambda_n = j_{\nu} ^2$, i.e., $\Delta w + j_{\nu} ^2 w = 0$, we use integration by parts and get that 
        \begin{align}
            \int_{B_{1}} |w(x)|^2 \de x = j_{\nu} ^{-2} \int_{B_{1}} |\nabla w(x)|^2 \de x = \frac{1}{2} (J_{\nu} ' (j_{\nu}))  ^2 \in \left[  \frac{1}{2} \nu^{-\frac{4}{3}}, \nu^{-\frac{4}{3}} \right].
        \end{align}

    Next, we estimate $ \int_{B_{\zeta_n}} |\nabla w(x)|^2 \de x$ for large $\nu$. By \Cref{lem:gradient L2 norm}, we estimate the three terms separately. First, by \Cref{lem:supercritical region}, when $\nu$ is large enough,
        \begin{align}
            \frac{j_{\nu} ^2 \zeta_n ^2 - \nu ^2}{2} J_{\nu} ^2 (j_{\nu} \zeta_n) = \frac{-2 \nu^{\frac{5}{3}} + \nu^{\frac{4}{3}}}{2} J_{\nu} ^2 (\nu - \nu ^{\frac{2}{3}}) \leq \nu^{\frac{4}{3}} 2^{-\frac{2}{3}\nu ^{\frac{1}{3}}} < 2^{-\frac{1}{2}\nu ^{\frac{1}{3}}} .
        \end{align}
    Next,
        \begin{align}
            \frac{j_{\nu} ^ 2 \zeta_n^2 }{2} \left( J_{\nu} ' (j_{\nu} \zeta_n) \right) ^2 = \frac{(\nu - \nu ^{\frac{2}{3}})^2}{2} \left( J_{\nu} ' (\nu - \nu ^{\frac{2}{3}}) \right) ^2 \leq \nu ^2 \cdot \nu^{-1} 2^{-\frac{2}{3}\nu ^{\frac{1}{3}}} < 2^{-\frac{1}{2}\nu ^{\frac{1}{3}}} .
        \end{align}
    Finally, because $-D J_{\nu} ^2(j_{\nu} \zeta_n) \leq 0$,
        \begin{align}
            \begin{split}
                &\left(-D J_{\nu}(j_{\nu} \zeta_n) +  j_{\nu} \zeta_n J_{\nu} ' (j_{\nu} \zeta_n)\right) J_{\nu}(j_{\nu} \zeta_n) \leq  (\nu - \nu ^{\frac{2}{3}}) J_{\nu} ' (\nu - \nu ^{\frac{2}{3}}) J_{\nu}(\nu - \nu ^{\frac{2}{3}})
                \\  &\leq \nu \cdot \nu ^{-\frac{1}{2}} 2^{-\frac{2}{3}\nu ^{\frac{1}{3}}} < 2^{-\frac{1}{2}\nu ^{\frac{1}{3}}}.
            \end{split}
        \end{align}
    Hence, for $\nu$ large enough, we have that 
        \begin{align}
            \int_{B_{\zeta_n}} |\nabla w(x)|^2 \de x \leq  3 \cdot 2^{-\frac{1}{2}\nu ^{\frac{1}{3}}} < 2^{-\frac{1}{3}\nu ^{\frac{1}{3}}}.
        \end{align}
\end{proof}

In \Cref{lem:gradient L2 whispering}, because $\zeta_n \to 1^-$ as $n \to +\infty$, \Cref{lem:gradient L2 whispering} already illustrates the Whispering Gallery Mode for the Dirichlet energy of $w$. Next, we build up similar results for $L^p$-norms of $w$, which are similar to \cite{NG2013}. We need the following lemma for Bessel functions. See also for example the equation (9.1.62) in \cite{AS1964} or the equation (10.14.4) in \cite{NIST:DLMF} for \Cref{lem:r<1 bessel function}.
\begin{lemma}\label{lem:r<1 bessel function}
    For any $x \in \bR$,
        \begin{align}
            |J_{\nu}(x)| \leq \frac{\left(\frac{x}{2}\right)^{\nu}}{ \Gamma(\nu +1)},
        \end{align}
    where $\Gamma(\cdot)$ is the Gamma function.
\end{lemma}

\begin{lemma}\label{lem:Lp whispering}
    Fix $p \geq 2$ and let $\zeta_n = j_{\nu} ^{-1} (\nu - \nu ^{\frac{2}{3}}) < 1$. We have that for $\nu$ large enough (depending on $d$), for any $m \in \llbracket 1, N_n \rrbracket$,
    \begin{align}
        \int_{B_{\zeta_n}} |w(x)|^{p} \de x \leq 2^{-\frac{1}{6}{\nu}^{\frac{1}{3}} \cdot p} .
    \end{align}
\end{lemma}
\begin{proof}
    For any $s \geq 0$, we first see that 
        \begin{align}
            \int_{B_{s}} |w(x)|^p \de x = \int_0 ^s r^{d-1 - pD} |J_{\nu}(j_{\nu} r)|^p \de r \int_{\bS^{d-1}} |Y (\theta)| ^p \de \theta.
        \end{align}
    We first compute $\int_{\bS^{d-1}} |Y (\theta)| ^p \de \theta$. Recall that $\int_{\bS^{d-1}} |Y(\theta)| ^2 \de \theta = 1$ and $\Delta_{\bS^{d-1}} Y + n(n+d-2) Y = 0$. We also notice that $\nu^2 = n^2+\frac{d^2}{4}+1 +nd - 2n - d = n(n+d-2)+(\frac{d}{2}-1)^2 \geq n(n+d-2)$. We use \Cref{lem:sogges estimate} and get that for some $C>0$ depending on $d$,
        \begin{align}\label{eq:Ynm lp norm}
            \|Y\|_{L^{p} (\bS^{d-1})} \leq  C \left(n(n+d-2)\right) ^{\frac{d-1}{4}} \cdot \|Y\|_{L^{2} (\bS^{d-1})}  \leq C \nu ^{\frac{d-1}{2}}.
        \end{align}

    Next, we estimate the term $\int_0 ^{\zeta_n} r^{d-1 - pD} |J_{\nu}(j_{\nu} r)|^p \de r$. By \Cref{lem:j_mu asymptotics}, we have that
        \begin{align}
            \begin{split}
                \int_0 ^{\zeta_n} r^{d-1 - p D} |J_{\nu}(j_{\nu} r)|^{p} \de r &= j_{\nu} ^{p D - d} \int_0 ^{(\nu -  \nu ^{\frac{2}{3}})} r^{d-1 - p D} |J_{\nu}(r)|^{p} \de r 
                \\  &\leq 2^{|p D-d|} \nu^{p D - d} \int_0 ^{(\nu -  \nu ^{\frac{2}{3}})} r^{d-1 - p D} |J_{\nu}(r)|^{p} \de r,
            \end{split}
        \end{align}
and we divide the above integral into the integrals over interval $[0,1]$ and $[1, \nu - \nu^{\frac{2}{3}}]$. Using \Cref{lem:r<1 bessel function}, we have that
    \begin{align}\label{eq:zeta lp norm 0 1}
        \begin{split}
            \nu^{pD - d} \int_0 ^{1} r^{d-1 - pD} |J_{\nu}(r)|^{p} \de r &= \nu^{pD - d} \int_0 ^{1} r^{d-1} |r^{-D}J_{\nu}(r)|^{p} \de r 
            \\  &\leq \nu^{pD - d} \int_0 ^{1}  \left|\frac{r^n}{2^{\nu} \cdot \Gamma(\nu+1)} \right|^{p} \de r \leq \frac{\nu^{pD - d}}{2^{\nu p} \cdot (\Gamma(\nu+1))^p},
        \end{split}
    \end{align}
which goes to $0$ faster than $2^{-\nu p}$, because $\Gamma(\nu+1) \sim \nu!$, which is much larger than $\nu^{D - d/p}$.
By \Cref{lem:supercritical region}, we have that
    \begin{align}\label{eq:zeta lp norm 1 nu}
        \begin{split}
            &\nu^{pD - d} \int_1 ^{(\nu -  \nu ^{\frac{2}{3}})} r^{d-1 - pD} |J_{\nu}(r)|^{p} \de r \leq 2^{-\frac{1}{3}{\nu}^{\frac{1}{3}} \cdot p} \cdot \nu^{p D - d} \int_1 ^{(\nu -  \nu ^{\frac{2}{3}})} r^{d-1 - pD} \de r 
            \\  &= 2^{-\frac{1}{3}{\nu}^{\frac{1}{3}} \cdot p} \cdot  \int_{\nu^{-1}} ^{(1 -  \nu ^{-\frac{1}{3}})} r^{d-1 - pD} \de r \leq 2^{-\frac{1}{3}{\nu}^{\frac{1}{3}} \cdot p} \cdot  \int_{\nu^{-1}} ^{1} r^{d-1 - pD} \de r.
        \end{split}
    \end{align}
When $d-1-pD < 0$, we see that $\int_{\nu^{-1}} ^{1} r^{d-1 - pD} \de r < \nu^{p D +1- d} \int_{\nu^{-1}} ^{1} \de r < \nu^{p D +1- d}$. When $d-1-pD \geq 0$, $\int_{\nu^{-1}} ^{1} r^{d-1 - pD} \de r \leq \int_{\nu^{-1}} ^{1} 1 \de r <1$. Hence, in all these cases, \eqref{eq:zeta lp norm 1 nu} goes to $0$ exponentially fast of order at least $2^{-\frac{1}{4}{\nu}^{\frac{1}{3}} \cdot p}$ when $\nu$ is large enough (depending on $d$). Combine \eqref{eq:Ynm lp norm}, \eqref{eq:zeta lp norm 0 1}, and \eqref{eq:zeta lp norm 1 nu}, we finish the proof for \Cref{lem:Lp whispering}.

\end{proof}

\begin{proof}[Proof of \Cref{thm:main whispering gallery linear eigen}]
    \Cref{thm:main whispering gallery linear eigen} actually follows directly from \Cref{lem:gradient L2 whispering} and \Cref{lem:Lp whispering}. To show the exact results in \Cref{thm:main whispering gallery linear eigen}, we explain those terms in \Cref{thm:main whispering gallery linear eigen}. First, because all Dirichlet Laplacian eigenfunctions on $\Omega = B_1$ are linear combinations of \eqref{e:sharp eigenfunction}, there must be an $n$ such that $\Lambda = \Lambda_n = j_{\nu}^2$.

    First, we notice that $\zeta_n > \tau_n = 1-2\Lambda_n ^{-\frac{1}{6}} = 1-2j_{\nu} ^{-\frac{1}{3}}$ when $\nu$ is large, which is equivalent to $\nu - \nu^{\frac{2}{3}} > j_{\nu} - 2j_{\nu}^{\frac{2}{3}} \iff 2j_{\nu}^{\frac{2}{3}} > \nu^{\frac{2}{3}}+ j_{\nu} - \nu $. Using \Cref{lem:j_mu asymptotics}, we see that when $\nu$ is large enough, $2j_{\nu}^{\frac{2}{3}} > 2 \nu^{\frac{2}{3}}$, and $\nu^{\frac{2}{3}}+ j_{\nu} - \nu < \nu^{\frac{2}{3}} + 2\nu^{\frac{1}{3}} < 2 \nu^{\frac{2}{3}}$. Hence, $\zeta_n > \tau_n$, and the results in \Cref{lem:gradient L2 whispering} and \Cref{lem:Lp whispering} hold true also after replacing $\zeta_n$ with $\tau_n$.

    By \Cref{lem:Lp whispering} and \Cref{lem:gradient L2 whispering}, we see that when $\nu$ is large enough,
        \begin{align}
            \left(\int_{B_{\tau_n}} |w(x)|^{p} \de x \right)^{\frac{1}{p}} \leq 2^{-\frac{1}{6}{\nu}^{\frac{1}{3}}} \leq 2^{-\frac{1}{6}{\nu}^{\frac{1}{3}}}  \cdot \left(2\nu^{\frac{4}{3}} \cdot \int_{B_1} |w(x)|^{2} \de x \right)^{\frac{1}{2}} \leq 2^{-\frac{1}{8}{\nu}^{\frac{1}{3}}} \left(\int_{B_1} |w(x)|^{2} \de x \right)^{\frac{1}{2}}.
        \end{align}
    Also, we have that $2^{-\frac{1}{8}{\nu}^{\frac{1}{3}}} \leq 2^{-\frac{1}{10}{\Lambda}_n ^{\frac{1}{6}}}$ when $\nu$ is large enough because by \Cref{lem:j_mu asymptotics}, $\nu \asymp j_\nu \asymp \Lambda_n ^{\frac{1}{2}}$ when $\nu$ is large enough.
    The above inequality is invariant under the scaling of $w$. Hence, because $\|w_{\Lambda}\|_2 = 1$ for the $w_{\Lambda}$ in \Cref{thm:main whispering gallery linear eigen}, this proves the second inequality in \eqref{eq:main linear eigen 1}. The first inequality in \eqref{eq:main linear eigen 1} can be proved by the two inequalities in \Cref{lem:gradient L2 whispering} in a similar way.
\end{proof}

%--------------------------------
%\vspace{0.5in}
%\noindent {\bf %Acknowledgments} 

%-------------------------------------------------

%%%% Main text entry area:

%\section*{???}%% if no title is needed, leave empty \section*{}.
%\end{appendix}
%%%%%%%%%%%%%%%%%%%%%%%%%%%%%%%%%%%%%%%%%%%%%%
%% Multiple Appendixes:                     %%
%%%%%%%%%%%%%%%%%%%%%%%%%%%%%%%%%%%%%%%%%%%%%%

%------------------------------------------------

%\bibliographystyle{acm} % Style BST file (imsart-number.bst or imsart-nameyear.bst)
%\bibliography{filename.bib} 

\bibliographystyle{abbrv} % Style BST file (imsart-number.bst or imsart-nameyear.bst)
\bibliography{whispering_gallery.bib}

\end{document}